\documentclass[11pt,reqno]{amsart}

\usepackage{amssymb}
\usepackage{amscd}
\usepackage{amsfonts}
\usepackage{mathrsfs}
\usepackage{setspace}
\usepackage{version}
\usepackage{mathtools}
\usepackage[normalem]{ulem}
\usepackage{float}

\usepackage{pgfplots,wrapfig}
\usepackage{lipsum}

\usepackage{xcolor}

\usepackage{amsmath}
\usepackage[english]{babel}
\usepackage{amssymb}
\usepackage{enumitem}
\usepackage[initials]{amsrefs}
\usepackage[all]{xy}
\usepackage{bbm}

\SelectTips{cm}{}

\usepackage[pdftex,colorlinks,citecolor=blue]{hyperref}

\usepackage{graphicx}
\usepackage{tikz-cd}

\newcommand{\bbN}{{\mathbb N}}

\newcommand{\bbR}{{\mathbb R}}
\newcommand{\bbZ}{{\mathbb Z}}

\newcommand{\bbH}{\mathbb{H}}
\newcommand{\bbF}{\mathbb{F}}

\newcommand{\id}{\operatorname{id}}

\newcommand{\vol}{\operatorname{vol}}
\newcommand{\ind}{\operatorname{ind}}
\newcommand{\proj}{\operatorname{proj}}

\newcommand{\SL}{\operatorname{SL}}
\newcommand{\charac}{\operatorname{char}}
\newcommand{\Aut}{\operatorname{Aut}}

\newcommand{\PSL}{\operatorname{PSL}}

\newcommand{\GL}{\operatorname{GL}}
\newcommand{\SO}{\operatorname{SO}}

\newtheorem{mthm}{Theorem}

\newtheorem{theorem}{Theorem}[section]
\newtheorem{lemma}[theorem]{Lemma}

\newtheorem{corollary}[theorem]{Corollary}
\newtheorem{cor}[theorem]{Corollary}
\newtheorem{proposition}[theorem]{Proposition}

\theoremstyle{definition}

\newtheorem{example}[theorem]{Example}

\newtheorem{remark}[theorem]{Remark}

\numberwithin{equation}{section}

\newcommand{\efface}[1]{}
\pgfplotsset{compat=1.14} 

\begin{document}

\title[(Non)-escape of mass and equidistribution on trees]{(Non)-escape of mass and equidistribution for horospherical actions on trees}

\author{Corina Ciobotaru}
\author{Vladimir Finkelshtein}
\author{Cagri Sert}
\address{Department of Mathematics, University of Fribourg, Chemin du Mus\'{e}e 23, 1700 Fribourg, Switzerland}
\email{corina.ciobotaru@gmail.com}
\address{Mathematisches Institut, Georg-August-Universit\"{a}t G\"{o}ttingen, Bunsenstra\ss e 3-5, 37073 G\"{o}ttingen, Germany}
\email{filyok@gmail.com}
\address{Institut f\"{u}r Mathematik, Universit\"{a}t Z\"{u}rich, 190, Winterthurerstrasse, 8057 Z\"{u}rich, Switzerland}
\email{cagri.sert@math.uzh.ch}

\begin{abstract}
Let $G$ be a large group acting on a biregular tree $T$ and $\Gamma \leq G$ a geometrically finite lattice.
In an earlier work, the authors classified orbit closures of the action of the horospherical subgroups on $G/\Gamma$. In this article we show that there is no escape of mass and use this to prove that, in fact, dense orbits equidistribute to the Haar measure on $G/\Gamma$. On the other hand, we show that new dynamical phenomena for horospherical actions appear on quotients by non-geometrically finite lattices: we give examples of non-geometrically finite lattices where an escape of mass phenomenon occurs and where the orbital averages along a F\o lner sequence do not converge. In the last part, as a by-product of our methods, we show that projections to $\Gamma \backslash T$ of the uniform distributions on large spheres in the tree $T$ converge to a natural probability measure on $\Gamma \backslash T$. Finally, we apply this equidistribution result to a lattice point counting problem to obtain counting asymptotics with exponential error term.
\end{abstract}


\maketitle

\section{Introduction}
Let $T$ be a $(d_1,d_2)$-biregular tree with $d_1,d_2\geq 3$. Denote by $\Aut(T)$ the group of automorphisms acting without edge inversion. Let $G$ be a non-compact, closed subgroup of $\Aut(T)$ acting transitively on the boundary of the tree $\partial T$. Let $\Gamma \leq G$ be a lattice and $X=G/\Gamma$. 
 
This parallels the classical setting of homogeneous dynamics, where one studies the actions of certain subgroups on a quotient of a linear algebraic group by a lattice. These two worlds intersect, for example, when $G=\SL_2(k)$, where $k$ is a non-archimedean local field, in which case $G$ naturally acts on the associated Bruhat--Tits tree. However, our geometric setting also comprises many groups $G\leq \Aut(T)$, including $\Aut(T)$ itself, that are not linear \cite{caprace-reid-willis}.


We first focus on the homogeneous space $X=G/\Gamma$, where $\Gamma$ is a geometrically finite lattice. The dynamics of discrete geodesic flow on $X$ was considered by Paulin in \cite{paulin.cont.frac}, and is related, among others, to the theory of continued fractions in non-archimedean local fields. We recall that when $G$ is linear, by works of Raghunathan and Lubotzky \cite{raghunathan.geofin,lubotzky.gafa}, any lattice therein is geometrically finite.

In our geometric setup, the role of Ad-unipotent subgroups in classical homogeneous dynamics is played by the horospherical subgroups $G_\eta^0$ of $G$, for $\eta \in \partial T$. In the earlier work \cite{CFS}, the authors classified Borel probability measures on $G/\Gamma$ invariant under $G_\eta^0$-action for large class of groups $G$ and general lattices $\Gamma$, establishing an analogue of Dani's result in \cite{dani.classification}. Moreover, it was shown that when $\Gamma$ is geometrically finite, $G_\eta^0$-orbits are either compact or dense, as in the classical result of Hedlund \cite{hedlund} on the horocycle flow on finite volume hyperbolic surfaces.

\subsection{Non-escape of mass}\label{subsec.intro.equidist}

The horospherical group $G_\eta^0$ is amenable and one can easily construct F{\o}lner sequences therein: let $a \in G$ be a hyperbolic element that has $\eta$ as its attracting fixed point on $\partial T$ and let $M$ be the compact subgroup of $G^0_\eta$ that fixes pointwise the translation axis of $a$ in $T$. Then for any $M$-invariant compact subset $O$ with non-empty interior in $G^0_\eta$, the sequence $(O_t:=a^tOa^{-t})_{t \in \mathbb{N}}$ constitutes a F{\o}lner sequence in $G^0_\eta$ (see e.g.\ \cite[Lemma 2.10]{CFS}). In the sequel, we shall refer to such sequences $O_t$ as good F{\o}lner sequences. F{\o}lner sequences allow one to average along larger and larger pieces of the orbits. For $x\in X$, we define $\nu_{x,t} = m_{O_t} * \delta_x$, where $m_{O_t}$ is the normalized restriction of the Haar measure $m_{G^0_\eta}$ to $O_t$; in other words for $f \in C_c(X)$,
\begin{equation*}
 \int_X f(y)  d\nu_{x,t}(y) = \int_{O_t} f(ux) dm_{O_t}(u).
\end{equation*}

The probability measures $\nu_{x,t}$ are called the orbital measures. 

In general, one can have a qualitative information on the statistical behaviour of the \textit{typical} points $x \in X$. This can be done using the Howe--Moore property, established in our setting in \cite{burger-mozes.product} and amenable ergodic theorem \cite{lindenstrauss}. Our topological result in \cite{CFS} says, however, that \textit{all} points $x \in X$ that do not lie in a compact $G^0_\eta$-orbit have dense orbits. Therefore, the immediate question arises whether every dense orbit equidistributes to the Haar measure on $G/\Gamma$. First possible obstruction to this is the escape of mass phenomenon. Our first result states that this does not happen when $\Gamma$ is a geometrically finite lattice. 

\begin{mthm}[Non-escape of mass]\label{thm.dm.geofin}
Let $T$ be a $(d_1,d_2)$-biregular tree, with $d_1,d_2\geq 3$, and $G$ a non-compact, closed subgroup of $\Aut(T)$ acting transitively on $\partial T$. Let $\Gamma$ be a geometrically finite lattice in $G$, $\eta \in \partial T$ and $O_t$ a good F\o lner sequence in $G^0_\eta$. Then, for every $\varepsilon>0$, there exists a compact set $K=K(\varepsilon) \subset X$ such that for every $x \in X$ not contained in a compact $G^0_\eta$-orbit, there exists a positive integer $N=N(x,\varepsilon)$ with the property that for every $t \geq N$, we have
\begin{equation}\label{eq.dm.rec}
\nu_{x,t}(K) > 1-\varepsilon.
\end{equation}
\end{mthm}

The above is known as non-escape of mass. In the context of one-parameter unipotent flows on quotients of real Lie groups, it is due to Dani and Margulis \cite{dani.recurrence, dani-margulis}. Our result also applies to the linear setting, we now describe this special case. Let $k$ be a non-archimedean local field and $H$ be the group of $k$-points of a connected semisimple linear algebraic $k$-group $\mathbb{H}$ of $k$-rank one. Let $\mathbb{A}$ be a maximal $k$-split torus in $\mathbb{H}$, $\mathbb{Z}$ its centralizer, $\mathbb{U}$ a maximal unipotent subgroup, $\mathbb{P}$ the normalizer of $\mathbb{U}$, and, respectively, $A,Z, U, P$ be the groups of $k$-points. The group $H$ acts by automorphisms \cite{bruhat-tits} (see also \cite[page 411]{lubotzky.gafa}) on its Bruhat--Tits building which is a bi-regular tree $T$. If $\mathbb{H}$ is simply connected, then $H$ embeds as a closed subgroup of $\Aut(T)$. In general, $H$ might have edge inversion and in this case, we shall replace it with an index two subgroup that acts without edge inversion. Moreover, let $K$ be a good maximal compact group of $H$. The group $K$ is the stabilizer of a vertex of $T$, $P=ZU$ is the stabilizer of an end $\eta \in \partial T$ and we have the Iwasawa decomposition $H=KP$ (see \cite[\S 4]{bruhat-tits} or \cite[\S 8.2.1]{benoist-quint.book}).  Finally let $M$ be the compact subgroup $K \cap Z$ of $H$. In our geometric setting, we have $H^0_\eta=MU$ and the following result is an immediate consequence of the previous theorem:

\begin{corollary}\label{corol.non-escape}
Let $H$ and its subgroups $M,U$ be as above. Let $\Lambda$ be a lattice in $H$ and $O_t$ be a good F{\o}lner sequence in $MU$. Then, for every $\varepsilon>0$, there exists a compact set $K=K(\varepsilon) \subset X$ such that for every $x \in X$ not contained in a compact $MU$-orbit, there exists a positive integer $N=N(x,\varepsilon)$ with the property that for every $t \geq N$, we have $
\nu_{x,t}(K) > 1-\varepsilon$.
\end{corollary}
This corollary is only relevant for fields $k$ with $\charac k \neq 0$. Indeed in the zero characteristic case, by a result of Tamagawa \cite{tamagawa} (also observed in \cite{ratner:p-adic}), every lattice in $H$ is uniform. 
We also remark that the version of the previous corollary for $U$ (instead of $MU$) holds as well. Finally, we note that a related result which would imply the previous corollary was mentioned in \cite[page 467]{ghosh}.

An immediate general consequence of Theorem \ref{thm.dm.geofin} is

\begin{cor}\label{cor:convergence-of-orbit-measures}
For every $x \in X$, every weak-$\ast$ limit of the sequence $\nu_{x,t}$ as $t\to \infty$ is a $G_\eta^0$-invariant probability measure on $X$.
\end{cor}

In the proof of Theorem \ref{thm.dm.geofin}, exploiting the underlying geometric setting, we translate the problem of understanding the  distribution of $G^0_\eta$-orbit in $G/\Gamma$ to the language of Markov chains, where it appears as a problem of controlling the distributions of a Markov chain with changing starting distributions. We then rely on two main ingredients: the first is a qualitative description of the behaviour of the discrete geodesic flow on $G/\Gamma$, as studied in \cite{CFS}. This allows us to understand the behaviour of starting distributions of the Markov chain. The second ingredient is, naturally, a set of Markov chain theoretical tools. The proof is then carried out by combining the two ingredients.

\subsection{Equidistribution of orbits}
For example, when $G=\Aut(T)$ and for $x \in X$ lying in a compact $G^0_\eta$-orbit, by standard arguments, all weak-$\ast$ limits of $\nu_{x,t}$ are $G_\eta^0$-invariant probability measures supported on the homogeneous orbit. This orbit supports a unique  $G^0_\eta$-invariant measure and, hence, $\nu_{x,t}$ equidistribute to the homogeneous measure supported on the orbit closure.

Under the additional topological simplicity assumption on $G$, our second result yields a complete qualitative description of statistical behaviour of every $x \in X$ not contained in a compact $G^0_\eta$-orbit and for such $x \in X$, it identifies the limit of $\nu_{x,t}$ as the Haar measure:
\begin{mthm}[Equidistribution]\label{thm.equidist}
Let $T$ be a $(d_1,d_2)$-biregular tree, with $d_1,d_2\geq 3$, and $G$ a non-compact, closed, topologically simple subgroup of $\Aut(T)$ acting transitively on $\partial T$. Let $\Gamma$ be a geometrically finite lattice in $G$ and $O_t$ be a good F\o lner sequence in $G^0_\eta$. Assume $x\in X$ does not belong to a compact $G^0_\eta$-orbit. Then, the orbital measures $\nu_{x,t}$ equidistribute to the normalized Haar measure $m_X$ as $t\to \infty$, in other words, for every $f \in C_c(X)$, we have
$$
\int_{O_{t}} f(ux) dm_{O_{t}}(u) \underset{t \to \infty}{\longrightarrow} \int_X f(y) dm_X(y).
$$
\end{mthm}

The previous theorem has the following immediate consequence on the statistical behaviour of $G^0_\eta$-orbits. Let $L$ be a closed subgroup of $G$. A probability measure $\mu$ on $X$ is called $L$-homogeneous if it is the unique $L$-invariant probability measure on a closed $L$-orbit. It is said to be homogeneous if it is $L$-homogeneous for some closed subgroup $L<G$. A point $x \in X$ is called generic for $G^0_\eta$ (see \cite[Definition 1]{ratner:p-adic}) if for some (equivalently for any) good F\o lner sequence $O_t$, the sequence $\nu_{x,t}$ of orbital measures equidistributes to a homogeneous measure.
\begin{corollary}\label{corol.full}
Keep the hypotheses of Theorem \ref{thm.equidist}. 
Any $x\in X$ is generic for $G^0_\eta$.
\end{corollary}

In the context of unipotent flows on $\SL_2(\bbR)/\Gamma$, this result goes back to Dani--Smillie \cite{dani-smillie}. Since then, Ratner \cite{ratner.equidist,ratner.rpadic}, Shah \cite{shah.equidist} and others have obtained very general results in Lie groups or algebraic groups over local fields of characteristic zero, but even in the case of a semisimple linear group $G$ of rank one over a local field of positive characteristic, e.g.\ $\SL_2(k)$ with $k=\mathbb{F}_q((X^{-1}))$, this result does not appear in the literature. However, we remark that for arithmetic quotients of linear groups, one may deduce such an equidistribution result by combining the work of Mohammadi \cite{mohammadi} and the result mentioned by Ghosh in \cite[page 467]{ghosh}. In the linear setting, the previous results have the following immediate consequence: 

\begin{corollary}
Keep the hypotheses and the notation of Corollary \ref{corol.non-escape}. The statement of Corollary \ref{corol.full} holds when $G^0_\eta$ is replaced with the subgroup $MU$ of $H$. 
\end{corollary}

For example, for $H=\SL_2(\bbF_q((X^{-1})))$, one can take $\Gamma$ to be the non-uniform lattice $\SL_2(\bbF_q[X])$ and the groups $M$ and $U$ to be
\[
M=\begin{pmatrix}
\bbF_q[[X^{-1}]]^* & 0 \\
0 & \bbF_q[[X^{-1}]]^* 
\end{pmatrix},\,  \; \; \;   
 U = \begin{pmatrix}
1 &  \bbF_q((X^{-1})) \\
0 & 1
\end{pmatrix}.
\]
We remark that for uniform lattices, one can use Margulis' orbit thickening argument to show that $U$-action is uniquely ergodic (see Mohammadi \cite{mohammadi}, Ellis--Perrizo \cite{ellis-perrizo} or \cite[Lemma 6.3]{CFS}). It is also worth noting that for non-uniform quotients, using our geometric approach, one can show the version of the previous corollary for the $U$-action (instead of $MU$). Finally, we mention the work of Vatsal \cite{vatsal} in which the equidistribution results of Ratner \cite{ratner:p-adic,ratner.rpadic} for unipotent dynamics in the $p$-adic case were applicable with a geometric approach similar to ours (see \cite[Page 9-10]{vatsal}).

Regarding the proof of Theorem \ref{thm.equidist}, it is proven by using Theorem \ref{thm.dm.geofin}, the classification of $G_\eta^0$-orbits, given in \cite{CFS} and the Howe--Moore property established in \cite{burger-mozes.product}.

\subsection{New non-linear homogeneous dynamical phenomena} 
So far, the results obtained in Theorems \ref{thm.dm.geofin} and \ref{thm.equidist} for geometrically finite lattices parallel the more classical results in linear homogeneous dynamics. However, the family of tree lattices is very rich and, as opposed to the linear setting, there exist many non-geometrically finite lattices. These exhibit wilder behaviors than their linear counterparts giving rise to several interesting phenomena that do not appear in the classical setting. Various aspects  of these differences, as well as analogies, were studied by many, including Serre \cite{bass-serre}, Tits \cite{tits}, Bass--Kulkarni \cite{bass-kulkarni}, Burger--Mozes \cite{burger-mozes.local,burger-mozes.product}, Lubotzky \cite{lubotzky.gafa}, Bass--Lubotzky \cite{bass-lubotzky}, Paulin \cite{paulin.geo.fin}, Bekka--Lubotzky \cite{bekka-lubotzky} etc. The following results add a new \textit{dynamical aspect} to these non-linear phenomena showing that horospherical orbits on quotients by non-geometrically finite lattices can exhibit escape of mass, which does not occur in homogeneous dynamics in the linear setting.


\begin{mthm}[Escape of mass]\label{thm.escape}
For any $q\geq 2$, there exist a lattice $\Gamma$ in $G=\Aut(T_{2q+2})$ and $\eta \in \partial T_{2q+2}$ such that  for the trivial coset $x=e\Gamma \in X$, any compact $K\subset X$ and any good F\o lner sequence $(O_t)_{t \in \mathbb{N}}$ in $G^0_\eta$, we have
\[ \lim_{t\to \infty} \nu_{x,t}(K) = 0. \]
\end{mthm}

Recall that in the setting of unipotent dynamics on linear homogeneous spaces, by now classical results of Ratner \cite{ratner.class, ratner.equidist,ratner.rpadic}, Mozes, Shah \cite{mozes-shah, shah.equidist} and others show that the orbital averages along unipotent group actions always converge towards an invariant probability measure. The following result contrasts the classical situation by giving an example where we see not only an escape of mass phenomenon, but also a failure of convergence of the orbital averages along F\o lner sequences.



\begin{mthm}[Escape of mass and equidistribution]\label{thm.escape.and.equidist}
There exists a non-uniform lattice $\Gamma<\Aut(T_6)$ with the property that for any $\eta \in \partial T$ there exist points $x \in X=\Aut(T_6)/\Gamma$ such that for any good F\o lner sequence $(O_t)_{t\in \mathbb{N}}$ in $G^0_\eta$, the set of accumulation points of the sequence of orbital averages $\nu_{x,t}$ contains the zero measure and $m_X$.
\end{mthm}

The proof of this theorem is carried out in Section \ref{sec.escape} and consists of several parts. In fact, it yields an uncountable number of non-isomorphic such lattices in $\Aut(T_6)$. The construction of these lattices has a similar flavor as the constructions of Bass--Lubotzky in \cite{bass-lubotzky} to show that there are lattices of arbitrarily small covolumes in $\Aut(T)$. Once the candidate lattices are constructed, the escape of mass phenomenon is proven by exploiting further the aforementioned connection between the Markov chain theory and distributions of horospherical orbits. This step uses the relatively finer ingredient of subgaussian concentration estimates for geometrically ergodic Markov chains (see e.g.\ Dedecker-Gou\"{e}zel \cite{dedecker-gouezel}). Finally, the proofs of the uniqueness of the $G^0_\eta$-invariant probability measure and the equidistribution along some orbital averages rely, among others, on the mixing of the discrete geodesic flow and the positive recurrence of the associated Markov chain.



\subsection{Equidistribution of spheres} To describe the general problem that we study here, consider a morphism of graphs $\pi: T \to Q$, where $T$ is a biregular tree. For a vertex $\tilde{v} \in VT$, let $S(\tilde{v},n)$ be the set of vertices of $T$ at distance $n$ from $\tilde{v}$. Let $\rho_n$ be the uniform distribution on $S(\tilde{v},n)$. We are interested in the distributions $\pi_* \rho_n$ on $VQ$: do they have a limiting distribution and, if yes, can one identify it? Questions about equidistribution of spheres are well-studied in many homogeneous quotients: Euclidean spheres in $\bbR^d/\bbZ^d$ in  \cite{randol}
or hyperbolic spheres in quotients of hyperbolic space  $\mathbb{H}^d/\Gamma$, where $\Gamma$ is a lattice in $\SO(d,1)$  (see \cite[Theorem 3.3]{bekka-mayer}, \cite{randol} and \cites{shah-spheres, rudnick-duke-sarnak, eskin-mcmullen} for more general results with applications to various counting problems). In the following result, we answer such a question for the natural quotient $Q$ of the tree associated to the $\Gamma$-action, where $\Gamma$ is a general lattice in $\Aut(T)$.

\begin{mthm}[Equidistribution of spheres in quotients by tree lattices]\label{thm.dist.spheres}
Let $T$ be a biregular tree, $\Gamma \leq \Aut(T)$ a tree lattice. Denote by $Q=\Gamma \backslash T$. 
\begin{enumerate}
    \item (Non-escape of mass) For any $\epsilon>0$, there exists a finite subset $K\subset VQ$, such that for all $n \in \mathbb{N}$ we have
$$
\pi_* \rho_n(K) \geq 1-\epsilon.
$$
\item (Limiting distribution) There exists an integer $p$, and limiting probability distributions $\mu_0, ..., \mu_{p-1}$ on $VQ$ such that for all $v\in VT$ and for all $0\leq j <p$ we have 
\[   \pi_* \rho_{pn+j} \to \mu_j,  \quad \text{ as } n\to \infty .\]
    \item (Exponential convergence) If, in addition, $\Gamma$ is geometrically finite, we can take $p=2$ and there exists $r>1$ such that
    \[ \| \pi_* \rho_{2n+j} - \mu_j \| = o(r^{-n}),\]
where $\|.\|$ denotes the total variation norm.
\end{enumerate}
\end{mthm}
In geometrically finite case $(3)$, the measures $(\mu_j)_{j=0,1}$ coincide with the projection of the Haar measure $m_X$ by the natural map $\proj: \Aut(T)/\Gamma \to VQ$ by two different base points. The exponential rate of convergence $1/r$ in this result can be made effective, using the effective version of geometric ergodic theorem for Markov chains as in \cite{baxendale}.

The proof of the previous result relies on the tools we develop to prove Theorem \ref{thm.dm.geofin}. Indeed, the Markov chain that we construct to track the statistical behaviour of horospherical averages easily allows one to understand the spherical averages provided one proves a (positive) geometric recurrence property for (non-) geometrically finite lattice quotients. This is carried out in Section \ref{sec.spheres}. To draw an analogy, the overall proof can be seen to parallel, in considerably simpler fashion, the deduction of Theorem \cite[Theorem 4.4]{eskin-margulis-mozes} from Theorem 4.1 in that work.

\begin{remark}[Diophantine exponent vs.\ speed of equidistribution]\label{rk.intro.dioph}
In fact, in the geometrically finite case, using the geometric recurrence of the associated Markov chain (Lemma \ref{lemma:geom-erg}), one can show the version of the equidistribution in Theorem \ref{thm.equidist} on the quotient $VQ$ additionally with a speed as in (3) above. The equidistribution itself directly follows by projecting the measures $m_{O_t}$ and $m_X$ in Theorem \ref{thm.equidist} by the map $\proj$. The speed of equidistribution depends on a geometric diophantine exponent (see e.g.\ \cite[(1.6)]{strombergsson} and \cite{paulin.geo.fin, fishman-simmons-urbanski}) of the boundary point $g^{-1}\eta$ where $x=g\Gamma$. From this perspective, Theorem \ref{thm.dist.spheres} (3) can also be seen as a particular case based on the fact of hyperbolic geometry that large circles are well-approximated by horocycles \cite[p.116]{eskin-margulis-mozes} (see also Remark \ref{rk.equidist.in.X}).
\end{remark}


\subsection{Counting lattice points}
Another classical question closely related to the equidistribution of spheres is the problem of counting lattice points. To describe the general problem, consider a lattice $\Gamma$ (or more generally a discrete subgroup) in some locally compact topological group endowed with a non-negative functional $\|\cdot\|$. One is interested in describing the asymptotics of
$$  N(R) = | \{ \gamma \in \Gamma : \| \gamma \| \leq R  \} |.$$
This problem goes back to Gauss who was interested in the case $\bbZ^d \leq \bbR^d$ with Euclidean norm as the functional $|| \cdot ||$. This particular problem is known as Gauss circle problem and the sharp error rates are still unknown. For  $\Gamma \leq \SL_2(\bbR)$, one can take $|| \cdot ||$ to be the operator norm induced by the Euclidean norm on $\mathbb{R}^2$, in which case, we have $\| g\| = \exp(\frac{1}{2} d_{\bbH^2}(g.i, i))$. This was already studied by Delsarte \cite{delsarte} in 40's, who obtained the first non-euclidean counting results. In the same setting, lattice point counting problem is also closely related to the counting of closed geodesics on hyperbolic surfaces. For an extensive historical survey and overview of methods used, we refer to \cite{gorodnik-nevo}, where the authors also develop spectral techniques to study the lattice point counting problem in a large generality.

Coming back to our setting, in analogy with the real hyperbolic case, it is natural to consider the functional $\| g\| = d(g\tilde{o},\tilde{o})$, where
$\tilde{o}\in VT$ is some basepoint and $d$ the graph distance on the tree. Clearly, for a discrete $\Gamma$, $N(R)$ is finite and non-decreasing in $R$. The following result describes the growth asymptotics of $N(R)$ with exponential error term for a geometrically finite lattice $\Gamma$:

\begin{mthm}\label{thm.counting}
Let $T$ be a biregular tree, $\Gamma \leq \Aut(T)$ a geometrically finite tree lattice. Let $m$ be an Haar measure on $\Aut(T)$ and $m_X$ the induced finite measure on $X=\Aut(T)/\Gamma$. Fix a basepoint $\tilde{o}\in VT$ and for $R \in \mathbb{N}$, let 
$$  N(R) = | \{ \gamma \in \Gamma : d(\gamma\tilde{o}, \tilde{o}) \leq R  \} |.$$
Denote by $B_T(R)$ the cardinality of the set of vertices at an even distance from $\tilde{o}$ that is at most $R$.
Then, there exists $c \in (0,1)$ such that
$$ 
\left|\frac{N(2R)}{B_T(2R)} -  \frac{m(G_{\tilde{o}})}{m_X(X)} \right| < o(c^{2R}).
$$
\end{mthm}

We stress that unlike before, we do not normalize the measure $m_X$ to be a probability measure. We also remark that the main term $\frac{m(G_{\tilde{o}})}{m_X(X)}$ can alternatively be expressed as $(\sum_{v \in V'Q} \frac{1}{|\Gamma \cap G_{\tilde{v}}|} )^{-1}$, where for every vertex $v$ of $Q=\Gamma \setminus T$, $\tilde{v} \in VT$ denotes a lift of $v$, $G_{\tilde{v}}$ is the maximal compact subgroup of $\Aut(T)$ fixing $\tilde{v}$ and $V'Q$ denotes the set of vertices $Q$ at even distance from $\pi(\tilde{o})$. Finally, we note that $\Aut(T)$ acts without edge inversion and this implies that for every $g \in \Aut(T)$ and $\tilde{v} \in VT$, $d(g\tilde{v},\tilde{v}) \in \mathbb{N}$ is an even number. This is the reason why, in the previous statement, we only consider vertices at even distance from each other.

We remark that this theorem also follows from the main result of Kwon in \cite{kwon} and from the work of Roblin \cite[Chapitre 4, Corollaire 2]{roblin}. Our proof relies on our previous result on the equidistribution of spheres (Theorem \ref{thm.dist.spheres}) and is a relatively straightforward consequence thereof. An exponential error rate $c \in (0,1)$ can also be effectively calculated.

\bigskip

The article is organized as follows. We recall some preliminary material mostly on lattices in groups acting on trees and set our notation in \S \ref{sec.preliminary}. In \S \ref{sec.rec}, we associate a natural Markov chain to an edge-indexed graph, study its properties and use these to prove Theorem \ref{thm.dm.geofin} for geometrically finite lattices. In \S \ref{sec.equidist}, we prove Theorem \ref{thm.equidist}. Theorems \ref{thm.escape} and \ref{thm.escape.and.equidist} are proven in \S \ref{sec.escape}. In \S \ref{sec.spheres}, we introduce an auxiliary Markov chain and use this to study the edge-indexed graph associated to a general lattice and prove Theorems \ref{thm.dist.spheres} and \ref{thm.counting}.

\subsection*{Acknowledgements}
The authors are thankful to Marc Burger and Manfred Einsiedler for helpful discussions. The authors also thank an anonymous referee for a careful reading, several remarks clarifying the exposition and helpful bibliographical suggestions. V.F.\ is supported by ERC Consolidator grant 648329 (GRANT).  C.S.\ is supported by SNF grants 178958 and 182089.

\section{Preliminaries}\label{sec.preliminary}

\subsection{Basic notation}\label{subsec.basic}
We denote by $T$ a $(d_1,d_2)$-regular tree, with $d_1,d_2\geq 3$, with $VT$ its set of vertices and $ET$, its edges. All edges are directed and $\partial_0, \partial_1 : ET \to VT$ are, respectively, the initial and the terminal vertex maps. An (ordered) pair of edges $e_1,e_2$ is called \textit{consecutive} if $\partial_1(e_1)=\partial_0(e_2)$. A sequence of consecutive edges $e_1,...,e_n$ is called a \textit{path} of length $n$. We also refer to it as a path between $\partial_0(e_1)$ and $\partial_1(e_n)$. The distance $d(\cdot,\cdot)$ between two vertices of the graph is  defined as the minimal length of a path between these vertices. 



We denote by $\Aut(T)$ the group of tree automorphisms acting  without edge inversion, i.e.~the group of automorphisms $g$ such that $d(gv,v)=0 \pmod 2$ for one (equivalently every) vertex $v \in VT$. When $d_1=d_2$, this is an index two subgroup of full group of automorphisms. Endowed with pointwise convergence topology, it is a locally compact, second countable group. In this article $G$ always stands for a non-compact, closed subgroup of $\Aut(T)$ that acts transitively on the boundary $\partial T$ of $T$. 

Throughout the rest of the article, we fix a basepoint $\tilde{o}\in VT$ and a distinguished end $\eta \in \partial T$, and denote by $(y_0, y_1, y_2 ,...)$ the vertices of the infinite path converging to $\eta$, where $y_0=\tilde{o}$.

For a subset $S\subset T$, and a subgroup $H<\Aut(T)$, $H_S$ denotes the pointwise stabilizer of $S$ in $H$. Given $\eta \in \partial T$, we define
\[  G_\eta^0 : =\{ g\in G ~|~ \exists N  , \forall n\geq N \quad  g(y_n)=y_n   \}. \]
The group $G_\eta^0$ is called \textbf{the horospherical subgroup} (see \cite[Section 2]{CFS} for more details on horospherical subgroups). It is a closed and amenable subgroup of $G$ and as mentioned in the introduction, one can construct many good F\o lner sequences in $G^0_\eta$. The following sequence of compact open subgroups of $G_\eta^0$ yields a good and tempered F\o lner sequence that is particularly convenient for our geometric approach. For $t \in \mathbb{N}$, we set
\[  F_t := \{ g\in G_\eta^0 ~|~   g(y_t)=y_t   \}. \]
In fact, as we shall see, thanks to the structure of good F\o lner sequences, it will be sufficient to prove our results only for the sequence $F_t$. Denote by $m_G$ and $m_{G_\eta^0}$ the Haar measures on $G$ and $G_\eta^0$, respectively. By $m_{F_t}$ we denote the Haar probability measure on $F_t$ which clearly coincides with the normalized restriction of $m_{G^0_\eta}$ to $F_t$.

\subsection{Lattices and theirs associated edge-indexed graphs}\label{subsec.edge.indexed}

It is well-known that a subgroup $\Gamma \leq G$ is discrete if and only if all vertex stabilizers $\Gamma_v$ for $v\in VT$ are finite. A discrete subgroup $\Gamma \leq G$ is called a lattice if $X=G/\Gamma$ admits a $G$-invariant Borel probability  measure, in which case we denote this measure by $m_X$.  By our standing assumption of boundary transitivity of $G$, the quotient graph $G\backslash T$ has two vertices. Indeed by \cite[Lemma 3.1.1]{burger-mozes.local}, $G$ acts two-transitively on $\partial T$ which in turn implies that $G$ has precisely two orbits on $VT$. Moreover, since $G$ acts without edge inversions, it acts transitively on the set of vertices of even distance. In this case, $\Gamma$ is a lattice in $G$ if and only if it is a lattice in $\Aut(T)$. Therefore, all lattices we will consider are tree lattices, i.e.\ lattices in $\Aut(T)$. For convenience, we will often call them lattices without specifying the ambient group. We refer to \cite{bass-lubotzky} for more details on tree lattices and edge-indexed graphs.

Given a discrete subgroup $\Gamma$, there is a useful construction \cite{bass-kulkarni} of a graph $Q$ and map $\ind: EQ \to \mathbb{N}$ as follows: the graph $Q$ is the quotient graph $\Gamma \backslash T$, which is well-defined, since $\Gamma$ acts without edge inversion. Denote by $\pi:T\to Q$ the projection map. 
The index map $\ind:EQ \to \bbN$ is given by $ \ind(e) = [ \Gamma_{\partial_0(\tilde{e})}  :\Gamma_{\tilde{e}}  ]$, where $\tilde{e}\in ET$ is any edge with $\pi(\tilde{e})=e$. This clearly does not depend on the choice of the lift $\tilde{e}$. The pair $(Q,\ind)$ is called the \textbf{edge-indexed graph} associated to $\Gamma <\Aut(T)$.

For $v\in VQ$, we define $\deg(v)$ to be the valency of any of its lifts $\tilde{v}$. By definition of the map $\ind$, 
\begin{equation}\label{eq:valency-gen}
    \deg(v)=\sum_{e\in EQ : \partial_0(e)=v}\ind(e).
\end{equation}

Further, let $\Delta: EQ \to \bbR$ given by $\Delta(e)=\frac{\ind(\overline{e})}{\ind(e)}$ and for $u,v\in VQ$ define
\begin{equation}\label{eq.defn.N}
N_u (v) = \Delta(e_1)...\Delta(e_n),
\end{equation}
where $(e_1,...,e_n)$ is a path from $u$ to $v$. For an  an edge-indexed graph $(Q,\ind)$ associated with a discrete subgroup $\Gamma$, the value of $N_u(v)$ does not depend on the choice of the path. Fixing a basepoint $o\in VQ$ (for convenience, we use $o=\pi(\tilde{o})$),  the discrete subgroup $\Gamma$ is a lattice in $G$ (see \cite[\S 1.1.5]{bass-lubotzky}) if and only if

\begin{equation}\label{eq:volume-formula}
\vol_o(Q,\ind):= \sum_{v\in VQ} N_o(v)^{-1} < \infty,     
\end{equation}
where $d(.,.)$ denotes the graph distance on $Q$. We shall refer to this quantity as the volume of the edge-indexed graph $(Q,\ind)$ based at $o$. 
We also remark that changing the base point from $o$ to $o'$ has the effect of multiplying the previous sum by the rational number $\frac{\Delta(o')}{\Delta(o)}$, therefore does not affect its finiteness. 

Conversely, one can define an abstract edge-indexed graph $(Q,\ind)$ as a tuple consisting of a graph $Q$ and map $\ind: EQ \to \mathbb{N}$. Under natural assumptions on the associated maps $\Delta$ and $N$ as above, there exists a discrete subgroup $\Gamma$ whose associated edge-indexed graph coincides with $(Q,\ind)$ and the function $N$ is proportional to $v \mapsto |\Gamma_{\tilde{v}}|$, where $\tilde{v}$ is any lift of $v$ (see \cite[page 23]{bass-lubotzky} or \cite{bass-kulkarni}).



For a discrete group $\Gamma \leq G$, we define the projection map $\proj : G / \Gamma \to VQ$ by $\proj(g \Gamma):= \pi(g^{-1} \tilde{o})= \Gamma g^{-1}\tilde{o}$. The map $\proj$ is clearly continuous and has compact fibers in $G/\Gamma$: for each $v \in VQ$ and $g\in G$ such that $\proj(g \Gamma)=v$, we have $\proj^{-1}(v)= G_{\tilde{o}} g \Gamma$. 
Moreover, the measure of each fiber is 
\begin{equation*}
     m_X(G_{\tilde{o}}g\Gamma) = m_X(g^{-1}G_{\tilde{o}}g\Gamma)= m_X (G_{g^{-1}\tilde{o}}\Gamma) = m_X(G_{\tilde{o}}\Gamma)\frac{|\Gamma_{\tilde{o}}|}{|\Gamma_{g^{-1}\tilde{o}}|}
\end{equation*}
In other words, using the definition \ref{eq.defn.N} of the map $N_o$, we have 
\begin{equation}\label{eq:measure-of-po-fiber}
    \proj_*m_X(v) =\frac{1}{N_o(v)} \proj_*m_X(o).
\end{equation}

\subsection{Geometrically finite lattices}
Following \cite{bass-serre,bass-lubotzky}, we define a \textbf{Nagao ray} to be an edge-indexed graph $(Q,\ind)$ whose underlying graph $Q$ is an infinite ray and the map $\ind$ takes value $1$ on all edges directed towards the infinity except the edge emanating from the vertex $o$ at the origin. All edges $e$ directed away from infinity are indexed by $\deg(\partial_1(e))-1$. Here, an edge $e \in EQ$ is said to be directed towards infinity if $d(\partial_1(e),o)>d(\partial_0(e),o)$, and directed away from infinity otherwise. See Fig.~\ref{fig.nagao} for an example of Nagao ray in $(q_1+1,q_2+1)$-biregular tree. An \textbf{open Nagao ray} is obtained by removing the origin vertex from a Nagao ray.

\begin{figure}[h!]
\centering

\tikzset{every picture/.style={line width=0.75pt}} 

\begin{tikzpicture}[x=0.75pt,y=0.75pt,yscale=-1,xscale=1.4]

\draw    (101.85,131) -- (149.5,131) ;
\draw [shift={(99.5,131)}, rotate = 0] [color={rgb, 255:red, 0; green, 0; blue, 0 }  ][line width=0.75]      (0, 0) circle [x radius= 3.35, y radius= 3.35]   ;
\draw    (153.85,131) -- (201.5,131) ;
\draw [shift={(151.5,131)}, rotate = 0] [color={rgb, 255:red, 0; green, 0; blue, 0 }  ][line width=0.75]      (0, 0) circle [x radius= 3.35, y radius= 3.35]   ;
\draw    (205.85,131) -- (253.5,131) ;
\draw [shift={(203.5,131)}, rotate = 0] [color={rgb, 255:red, 0; green, 0; blue, 0 }  ][line width=0.75]      (0, 0) circle [x radius= 3.35, y radius= 3.35]   ;
\draw    (257.85,131) -- (305.5,131) ;
\draw [shift={(255.5,131)}, rotate = 0] [color={rgb, 255:red, 0; green, 0; blue, 0 }  ][line width=0.75]      (0, 0) circle [x radius= 3.35, y radius= 3.35]   ;
\draw    (309.85,131) -- (357.5,131) ;
\draw [shift={(307.5,131)}, rotate = 0] [color={rgb, 255:red, 0; green, 0; blue, 0 }  ][line width=0.75]      (0, 0) circle [x radius= 3.35, y radius= 3.35]   ;

\draw (393,131) node  [align=left] {. . .};
\draw (160,142) node  [align=left] {$\displaystyle 1$};
\draw (216,142) node  [align=left] {$\displaystyle 1$};
\draw (265,142) node  [align=left] {$\displaystyle 1$};
\draw (318,142) node  [align=left] {$\displaystyle 1$};
\draw (106,142) node  [align=left] {$\displaystyle *$};
\draw (137,144) node  [align=left] {$\displaystyle q_1$};
\draw (192,144) node  [align=left] {$\displaystyle q_2$};
\draw (241,144) node  [align=left] {$\displaystyle q_1$};
\draw (295,144) node  [align=left] {$\displaystyle q_2$};
\draw (347,144) node  [align=left] {$\displaystyle q_1$};

\end{tikzpicture}
\caption{Nagao ray, when $T$ is $(q_1+1,q_2+1)$-biregular. By convention, for edge $e$, the index $\ind(e)$ is written next to the vertex $\partial_0(e)$}
\label{fig.nagao}
\end{figure}
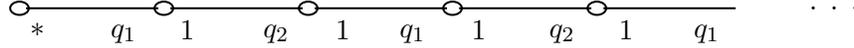

Following Paulin \cite{paulin.geo.fin}, a tree lattice $\Gamma$ is called \textbf{geometrically finite} if its associated edge-indexed graph $(Q,\ind)$ contains a finite subgraph $F$ whose set theoretic complement in $Q$ is a disjoint union of finitely many open Nagao rays. The \textbf{finite part} of $(Q,\ind)$ is the smallest non-empty finite subgraph $F$ with this property.
When $T$ is a $(q+1)$-regular tree, a tree lattice $\Gamma$ is called of \textbf{Nagao type} if the associated  edge-indexed graph is a Nagao ray (see \cite[Chapter 10]{bass-lubotzky}). Fig.~\ref{fig:nagaolattice} illustrates the corresponding edge-indexed graph. Another example of geometrically finite  lattice, where $T$ is $(3,10)$-biregular tree, is given in Fig.~\ref{fig:geom-fin}.

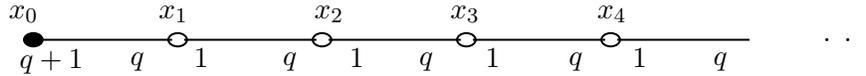
\begin{figure}[h!]
\centering

\tikzset{every picture/.style={line width=0.75pt}} 

\begin{tikzpicture}[x=0.75pt,y=0.75pt,yscale=-1,xscale=1.4]\label{fig:nagaoray}

\draw    (101.85,131) -- (149.5,131) ;
\draw [shift={(99.5,131)}, rotate = 0] [color={rgb, 255:red, 0; green, 0; blue, 0 }  ][fill={rgb, 255:red, 0; green, 0; blue, 0 }  ][line width=0.75]      (0, 0) circle [x radius= 3.35, y radius= 3.35]   ;
\draw    (153.85,131) -- (201.5,131) ;
\draw [shift={(151.5,131)}, rotate = 0] [color={rgb, 255:red, 0; green, 0; blue, 0 }  ][line width=0.75]      (0, 0) circle [x radius= 3.35, y radius= 3.35]   ;
\draw    (205.85,131) -- (253.5,131) ;
\draw [shift={(203.5,131)}, rotate = 0] [color={rgb, 255:red, 0; green, 0; blue, 0 }  ][line width=0.75]      (0, 0) circle [x radius= 3.35, y radius= 3.35]   ;
\draw    (257.85,131) -- (305.5,131) ;
\draw [shift={(255.5,131)}, rotate = 0] [color={rgb, 255:red, 0; green, 0; blue, 0 }  ][line width=0.75]      (0, 0) circle [x radius= 3.35, y radius= 3.35]   ;
\draw    (309.85,131) -- (357.5,131) ;
\draw [shift={(307.5,131)}, rotate = 0] [color={rgb, 255:red, 0; green, 0; blue, 0 }  ][line width=0.75]      (0, 0) circle [x radius= 3.35, y radius= 3.35]   ;
\draw (393,131) node  [align=left] {. . .};
\draw (150,118) node  [align=left] {$x_1$};
\draw (160,140) node  [align=left] {$1$};
\draw (216,140) node  [align=left] {$1$};
\draw (206,118) node  [align=left] {$x_2$};
\draw (265,140) node  [align=left] {$1$};
\draw (255,118) node  [align=left] {$x_3$};
\draw (318,140) node  [align=left] {$1$};
\draw (308,118) node  [align=left] {$x_4$};
\draw (106,142) node  [align=left] {$q+1$};
\draw (96,118) node  [align=left] {$x_0$};
\draw (137,142) node  [align=left] {$q$};
\draw (192,142) node  [align=left] {$q$};
\draw (241,142) node  [align=left] {$q$};
\draw (295,142) node  [align=left] {$q$};
\draw (347,142) node  [align=left] {$q$};

\end{tikzpicture}

\caption{Edge-indexed graph of a lattice of Nagao type. The index of $(x_0,x_1)$ is determined by \eqref{eq:valency-gen}. The finite part consists of the single vertex $x_0$. }
\label{fig:nagaolattice}
\end{figure}

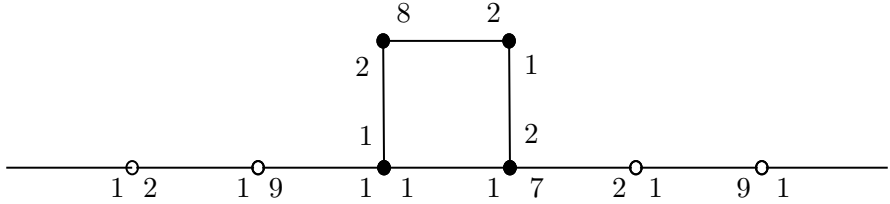
\begin{figure}[h!]
\tikzset{every picture/.style={line width=0.75pt}} 

    \centering

\begin{tikzpicture}[x=0.75pt,y=0.75pt,yscale=-1,xscale=0.87]

\draw    (507.85,182) -- (578.5,182) ;

\draw [shift={(505.5,182)}, rotate = 0] [color={rgb, 255:red, 0; green, 0; blue, 0 }  ][line width=0.75]      (0, 0) circle [x radius= 3.35, y radius= 3.35]   ;
\draw    (434.85,182) -- (503.15,182) ;
\draw [shift={(505.5,182)}, rotate = 0] [color={rgb, 255:red, 0; green, 0; blue, 0 }  ][line width=0.75]      (0, 0) circle [x radius= 3.35, y radius= 3.35]   ;
\draw [shift={(432.5,182)}, rotate = 0] [color={rgb, 255:red, 0; green, 0; blue, 0 }  ][line width=0.75]      (0, 0) circle [x radius= 3.35, y radius= 3.35]   ;
\draw    (361.85,182) -- (430.15,182) ;
\draw [shift={(432.5,182)}, rotate = 0] [color={rgb, 255:red, 0; green, 0; blue, 0 }  ][line width=0.75]      (0, 0) circle [x radius= 3.35, y radius= 3.35]   ;
\draw [shift={(359.5,182)}, rotate = 0] [color={rgb, 255:red, 0; green, 0; blue, 0 }  ][line width=0.75]      (0, 0) circle [x radius= 3.35, y radius= 3.35]   ;
\draw    (359.02,120.35) -- (359.48,179.65) ;
\draw [shift={(359.5,182)}, rotate = 89.55] [color={rgb, 255:red, 0; green, 0; blue, 0 }  ][line width=0.75] [fill={rgb, 255:red, 0; green, 0; blue, 0 }  ]     (0, 0) circle [x radius= 3.35, y radius= 3.35]   ;
\draw [shift={(359,118)}, rotate = 89.55] [color={rgb, 255:red, 0; green, 0; blue, 0 }  ][line width=0.75][fill={rgb, 255:red, 0; green, 0; blue, 0 }  ]      (0, 0) circle [x radius= 3.35, y radius= 3.35]   ;
\draw    (288.85,182) -- (357.15,182) ;
\draw [shift={(359.5,182)}, rotate = 0] [color={rgb, 255:red, 0; green, 0; blue, 0 }  ][line width=0.75]      (0, 0) circle [x radius= 3.35, y radius= 3.35]   ;
\draw [shift={(286.5,182)}, rotate = 0] [color={rgb, 255:red, 0; green, 0; blue, 0 }  ][line width=0.75]      (0, 0) circle [x radius= 3.35, y radius= 3.35]   ;
\draw    (215.85,182) -- (284.15,182) ;
\draw [shift={(286.5,182)}, rotate = 0] [color={rgb, 255:red, 0; green, 0; blue, 0 }  ][line width=0.75]      (0, 0) circle [x radius= 3.35, y radius= 3.35]   ;
\draw [shift={(213.5,182)}, rotate = 0] [color={rgb, 255:red, 0; green, 0; blue, 0 }  ][line width=0.75]  (0, 0) circle [x radius= 3.35, y radius= 3.35]   ;
\draw    (142.85,182) -- (211.15,182) ;
\draw [shift={(213.5,182)}, rotate = 0] [color={rgb, 255:red, 0; green, 0; blue, 0 }  ][line width=0.75]      (0, 0) circle [x radius= 3.35, y radius= 3.35]   ;
\draw [shift={(140.5,182)}, rotate = 0] [color={rgb, 255:red, 0; green, 0; blue, 0 }  ][line width=0.75]      (0, 0) circle [x radius= 3.35, y radius= 3.35]   ;
\draw    (286.48,179.65) -- (286.02,120.35) ;
\draw [shift={(286,118)}, rotate = 269.55] [color={rgb, 255:red, 0; green, 0; blue, 0 }  ][line width=0.75]  [fill={rgb, 255:red, 0; green, 0; blue, 0 }  ]    (0, 0) circle [x radius= 3.35, y radius= 3.35]   ;
\draw [shift={(286.5,182)}, rotate = 269.55] [color={rgb, 255:red, 0; green, 0; blue, 0 }  ][line width=0.75]  [fill={rgb, 255:red, 0; green, 0; blue, 0 }  ]    (0, 0) circle [x radius= 3.35, y radius= 3.35]   ;
\draw    (67.5,182) -- (140.5,182) ;

\draw    (288.35,118) -- (356.65,118) ;
\draw [shift={(359,118)}, rotate = 0] [color={rgb, 255:red, 0; green, 0; blue, 0 }  ][line width=0.75]      (0, 0) circle [x radius= 3.35, y radius= 3.35]   ;
\draw [shift={(286,118)}, rotate = 0] [color={rgb, 255:red, 0; green, 0; blue, 0 }  ][line width=0.75]      (0, 0) circle [x radius= 3.35, y radius= 3.35]   ;

\draw (132,192) node  [align=left] {$\displaystyle 1$};
\draw (151,192) node  [align=left] {$\displaystyle 2$};
\draw (205,192) node  [align=left] {$\displaystyle 1$};
\draw (444,192) node  [align=left] {$\displaystyle 1$};
\draw (518,192) node  [align=left] {$\displaystyle 1$};
\draw (224,192) node  [align=left] {$\displaystyle 9$};
\draw (276,192) node  [align=left] {$\displaystyle 1$};
\draw (300,192) node  [align=left] {$\displaystyle 1$};
\draw (276,166) node  [align=left] {$\displaystyle 1$};
\draw (274,131) node  [align=left] {$\displaystyle 2$};
\draw (298,104) node  [align=left] {$\displaystyle 8$};
\draw (350,104) node  [align=left] {$\displaystyle 2$};
\draw (372,165) node  [align=left] {$\displaystyle 2$};
\draw (372,130) node  [align=left] {$\displaystyle 1$};
\draw (350,192) node  [align=left] {$\displaystyle 1$};
\draw (375,192) node  [align=left] {$\displaystyle 7$};
\draw (423,192) node  [align=left] {$\displaystyle 2$};
\draw (495,192) node  [align=left] {$\displaystyle 9$};
\end{tikzpicture}

\caption{Edge-indexed graph associated with a geometrically finite lattice in $\Aut(T)$, where $T$ is $(3,10)$-biregular tree. The finite part contains  the solid nodes and the edges between them.}
    \label{fig:geom-fin}
\end{figure}


When $\Gamma$ is geometrically finite, we have a very useful characterization of compact $G_\eta^0$-orbits in $G/\Gamma$ (see \cite[Lemma 6.2]{CFS} or \cite[Proposition 3.1]{paulin.geo.fin}).

\begin{proposition}\label{prop.irrational}
Let $\Gamma\leq G$ geometrically finite lattice.  Let $g\in G$ be such that the $G_\eta^0$-orbit of $g\Gamma$ is not compact in $G/\Gamma$. Let $F$ denote the finite part of $Q=\Gamma \backslash T$. Then $\pi(g^{-1}y_t)$ belongs to $F$ for infinitely many values of $t$, in particular $t-d(\pi(g^{-1}y_t),F )$ is monotone non-decreasing and unbounded.
\end{proposition}

\subsection{Markov chains}

We recall some terminology and basic facts of the theory of Markov chains and set our notation. For more details, we refer the reader to \cite{DMPS,meyn-tweedie,LPW}.

Let $S$ be a countable set, and $P: S \times S \to [0,1]$ a Markov kernel, i.e.\ $\sum_{y \in S} P(x,y)=1$ for every $x \in S$. By (standard) abuse of notation, we shall also denote the associated Markov operator and its dual by $P$: for a function $f$ on $S$, $Pf(x)=\sum_{y}f(y)P(x,y)$. For a measure $\mu$ on $S$, $\mu P(\cdot)=\sum_{y} \mu(x)P(x,\cdot)$. 
For $n \in \mathbb{N}$, $P^n$ denotes the $n^{th}$-convolution power of $P$. For $s\in S$, we denote by $\delta_s$ the probability measure supported on $\{s\}$: for $s_1,s_2\in S$, $P^n(s_1,s_2):=\delta_{s_1}P^n(s_2)$.

The Markov kernel $P$ is called \textit{irreducible} if for every $s,t \in S$, there exists $n \in \mathbb{N}$ with $P^n(s,t)>0$. 
The \textit{period} of an irreducible Markov kernel $P$ is defined as $\gcd\{n \in \mathbb{N}\, |\, P^n(s,s)>0\}$ for some (or equivalently all) $s \in S$. If the period is $1$, the Markov chain is called \textit{aperiodic}. Denoting the period by $p$, there exists a partition $\Omega_0,\ldots,\Omega_{p-1}$ of the state space $S$ into \textit{cyclic classes} $\Omega_i$ such that for every $s \in \Omega_i$, $P(s,\Omega_{i+1})=1$ $(i \mod p)$. If $P$ is irreducible and has period $p$, then $P^p$ restricted to each cyclic class is irreducible and aperiodic. In a standard manner \cite[Section 3.1]{DMPS}, a Markov kernel yields a canonical Markov chain on the state space $S$. Therefore, we shall equivalently speak of a Markov chain being irreducible, aperiodic etc.\

A non-negative measure $\mu$ on $S$ is said to be \textit{stationary} for the Markov kernel $P$ if $\mu P=\mu$. 
An irreducible Markov kernel $P$ is called \textit{positive recurrent} if it admits a stationary probability measure, in which case this measure is unique.
If, moreover, $P$ has period $p$ then $\mu=\frac{1}{p}\sum_{i=0}^{p-1} \mu_{|\Omega_i}$ is a stationary measure of $P$, where $\mu_{|\Omega_i}$ is the unique stationary probability measure of $P^p$ restricted to $\Omega_i$. We also have $\mu_{\Omega_i} P=\mu_{\Omega_{i+1}}$ $(i \mod p)$.


For an irreducible aperiodic positive recurrent Markov chain and any initial distribution $\mu$, $\mu P^n $ converges to the stationary probability measure as $n\to \infty$. 
In case of an irreducible Markov chain that is not positive recurrent, $\mu P^n$ converges to $0$, regardless of the period.



\section{Non-escape of mass}\label{sec.rec}

The aim of this section is to prove Theorem \ref{thm.dm.geofin}. We start by associating a Markov chain with a tree lattice $\Gamma$, study its properties and eventually link the Markov chain to the study of orbital measures in $G/\Gamma$ of horospherical subgroups. 
If $\Gamma$ is a uniform lattice, there is nothing to prove in Theorem \ref{thm.dm.geofin}, so throughout the proof, $\Gamma$ is assumed to be non-uniform.

\subsection{The Markov chain}\label{subsec:mc-def}

Let $\Gamma$ be a tree lattice and $(Q,\ind)$ be the corresponding edge-indexed graph. Define the Markov chain $M_n$ with state space $EQ$ and transition probabilities given by 
\begin{equation*}
P(e_1,e_2) = \begin{cases}
 0 & \text{if } \partial_1(e_1) \neq \partial_0(e_2), \\
 \frac{\ind(e_2)-1}{\deg(\partial_1(e_1))-1} & \text{if }  e_2=\overline{e_1}, \\
 \frac{\ind(e_2)}{\deg(\partial_1(e_1))-1} & \text{otherwise. }
\end{cases}
\end{equation*}

Note that by \eqref{eq:valency-gen} transition probabilities sum to 1 so that $P$ is a Markov kernel. As the subsequent proofs will show, we are naturally led to the study of the Markov chain $M_n$ which can simply be seen as the image by quotient map $\pi$ of the simple random walk on the set of edges of the tree $T$. It came to our knowledge that this Markov chain was considered by Burger and Mozes \cite{bm-cat} in the study of the notion of divergence groups in $\Aut(T)$ and by Kwon \cite{kwon} in the study of mixing properties of the discrete geodesic flow. 


Let us illustrate the structure of this Markov chain as well as our subsequent use of it in a simple but important situation, that is when $\Gamma$ is lattice of Nagao type.

\begin{example}\label{ex:chain.on.nagao}
Let $\Gamma$ be Nagao lattice in $G\leq \Aut(T)$, where $T$ is a $(q+1)$-regular tree (see Fig.~\ref{fig:nagaolattice} for the corresponding edge-indexed graph). In this case, the above construction of Markov chain gives rise to a state space and transition probabilities as illustrated in Fig.~\ref{fig:nagaochain}. 

\begin{figure}[H]
    \centering

\tikzset{every picture/.style={line width=0.75pt}} 

\begin{tikzpicture}[x=0.75pt,y=0.75pt,yscale=-1,xscale=1]

\draw    (440.79,211.46) -- (535.75,211.46) ;
\draw [shift={(537.75,211.46)}, rotate = 180] [color={rgb, 255:red, 0; green, 0; blue, 0 }  ][line width=0.75]    (10.93,-3.29) .. controls (6.95,-1.4) and (3.31,-0.3) .. (0,0) .. controls (3.31,0.3) and (6.95,1.4) .. (10.93,3.29)   ;
\draw [shift={(440.79,211.46)}, rotate = 0] [color={rgb, 255:red, 0; green, 0; blue, 0 }  ][fill={rgb, 255:red, 0; green, 0; blue, 0 }  ][line width=0.75]      (0, 0) circle [x radius= 3.35, y radius= 3.35]   ;

\draw [shift={(537.79,211.46)}, rotate = 0] [color={rgb, 255:red, 0; green, 0; blue, 0 }  ][fill={rgb, 255:red, 0; green, 0; blue, 0 }  ][line width=0.75]      (0, 0) circle [x radius= 3.35, y radius= 3.35]   ;
\draw [shift={(149.79,109)}, rotate = 0] [color={rgb, 255:red, 0; green, 0; blue, 0 }  ][fill={rgb, 255:red, 0; green, 0; blue, 0 }  ][line width=0.75]      (0, 0) circle [x radius= 3.35, y radius= 3.35]   ;

\draw    (343.84,211.46) -- (438.79,211.46) ;
\draw [shift={(440.79,211.46)}, rotate = 180] [color={rgb, 255:red, 0; green, 0; blue, 0 }  ][line width=0.75]    (10.93,-3.29) .. controls (6.95,-1.4) and (3.31,-0.3) .. (0,0) .. controls (3.31,0.3) and (6.95,1.4) .. (10.93,3.29)   ;
\draw [shift={(343.84,211.46)}, rotate = 0] [color={rgb, 255:red, 0; green, 0; blue, 0 }  ][fill={rgb, 255:red, 0; green, 0; blue, 0 }  ][line width=0.75]      (0, 0) circle [x radius= 3.35, y radius= 3.35]   ;
\draw    (246.89,211.46) -- (341.84,211.46) ;
\draw [shift={(343.84,211.46)}, rotate = 180] [color={rgb, 255:red, 0; green, 0; blue, 0 }  ][line width=0.75]    (10.93,-3.29) .. controls (6.95,-1.4) and (3.31,-0.3) .. (0,0) .. controls (3.31,0.3) and (6.95,1.4) .. (10.93,3.29)   ;
\draw [shift={(246.89,211.46)}, rotate = 0] [color={rgb, 255:red, 0; green, 0; blue, 0 }  ][fill={rgb, 255:red, 0; green, 0; blue, 0 }  ][line width=0.75]      (0, 0) circle [x radius= 3.35, y radius= 3.35]   ;
\draw    (149.93,211.46) -- (244.89,211.46) ;
\draw [shift={(246.89,211.46)}, rotate = 180] [color={rgb, 255:red, 0; green, 0; blue, 0 }  ][line width=0.75]    (10.93,-3.29) .. controls (6.95,-1.4) and (3.31,-0.3) .. (0,0) .. controls (3.31,0.3) and (6.95,1.4) .. (10.93,3.29)   ;
\draw [shift={(149.93,211.46)}, rotate = 0] [color={rgb, 255:red, 0; green, 0; blue, 0 }  ][fill={rgb, 255:red, 0; green, 0; blue, 0 }  ][line width=0.75]      (0, 0) circle [x radius= 3.35, y radius= 3.35]   ;
\draw    (442.79,109) -- (537.75,109) ;
\draw [shift={(537.75,109)}, rotate = 0] [color={rgb, 255:red, 0; green, 0; blue, 0 }  ][fill={rgb, 255:red, 0; green, 0; blue, 0 }  ][line width=0.75]      (0, 0) circle [x radius= 3.35, y radius= 3.35]   ;
\draw [shift={(440.79,109)}, rotate = 0] [color={rgb, 255:red, 0; green, 0; blue, 0 }  ][line width=0.75]    (10.93,-3.29) .. controls (6.95,-1.4) and (3.31,-0.3) .. (0,0) .. controls (3.31,0.3) and (6.95,1.4) .. (10.93,3.29)   ;
\draw    (345.84,109) -- (440.79,109) ;
\draw [shift={(440.79,109)}, rotate = 0] [color={rgb, 255:red, 0; green, 0; blue, 0 }  ][fill={rgb, 255:red, 0; green, 0; blue, 0 }  ][line width=0.75]      (0, 0) circle [x radius= 3.35, y radius= 3.35]   ;
\draw [shift={(343.84,109)}, rotate = 0] [color={rgb, 255:red, 0; green, 0; blue, 0 }  ][line width=0.75]    (10.93,-3.29) .. controls (6.95,-1.4) and (3.31,-0.3) .. (0,0) .. controls (3.31,0.3) and (6.95,1.4) .. (10.93,3.29)   ;
\draw    (248.89,109) -- (343.84,109) ;
\draw [shift={(343.84,109)}, rotate = 0] [color={rgb, 255:red, 0; green, 0; blue, 0 }  ][fill={rgb, 255:red, 0; green, 0; blue, 0 }  ][line width=0.75]      (0, 0) circle [x radius= 3.35, y radius= 3.35]   ;
\draw [shift={(246.89,109)}, rotate = 0] [color={rgb, 255:red, 0; green, 0; blue, 0 }  ][line width=0.75]    (10.93,-3.29) .. controls (6.95,-1.4) and (3.31,-0.3) .. (0,0) .. controls (3.31,0.3) and (6.95,1.4) .. (10.93,3.29)   ;
\draw    (151.93,109) -- (246.89,109) ;
\draw [shift={(246.89,109)}, rotate = 0] [color={rgb, 255:red, 0; green, 0; blue, 0 }  ][fill={rgb, 255:red, 0; green, 0; blue, 0 }  ][line width=0.75]      (0, 0) circle [x radius= 3.35, y radius= 3.35]   ;
\draw [shift={(149.93,109)}, rotate = 0] [color={rgb, 255:red, 0; green, 0; blue, 0 }  ][line width=0.75]    (10.93,-3.29) .. controls (6.95,-1.4) and (3.31,-0.3) .. (0,0) .. controls (3.31,0.3) and (6.95,1.4) .. (10.93,3.29)   ;
\draw    (246.89,211.46) -- (246.89,111) ;
\draw [shift={(246.89,109)}, rotate = 450] [color={rgb, 255:red, 0; green, 0; blue, 0 }  ][line width=0.75]    (10.93,-3.29) .. controls (6.95,-1.4) and (3.31,-0.3) .. (0,0) .. controls (3.31,0.3) and (6.95,1.4) .. (10.93,3.29)   ;

\draw    (343.84,211.46) -- (343.84,111) ;
\draw [shift={(343.84,109)}, rotate = 450] [color={rgb, 255:red, 0; green, 0; blue, 0 }  ][line width=0.75]    (10.93,-3.29) .. controls (6.95,-1.4) and (3.31,-0.3) .. (0,0) .. controls (3.31,0.3) and (6.95,1.4) .. (10.93,3.29)   ;

\draw    (440.79,211.46) -- (440.79,111) ;
\draw [shift={(440.79,109)}, rotate = 450] [color={rgb, 255:red, 0; green, 0; blue, 0 }  ][line width=0.75]    (10.93,-3.29) .. controls (6.95,-1.4) and (3.31,-0.3) .. (0,0) .. controls (3.31,0.3) and (6.95,1.4) .. (10.93,3.29)   ;

\draw    (148.45,209.36) .. controls (124.6,175.19) and (122.6,146.14) .. (149.93,109) ;

\draw [shift={(149.93,211.46)}, rotate = 234.28] [color={rgb, 255:red, 0; green, 0; blue, 0 }  ][line width=0.75]    (10.93,-3.29) .. controls (6.95,-1.4) and (3.31,-0.3) .. (0,0) .. controls (3.31,0.3) and (6.95,1.4) .. (10.93,3.29)   ;
\draw    (151.22,110.58) .. controls (178.71,145.18) and (164.32,181.04) .. (149.93,211.46) ;

\draw [shift={(149.93,109)}, rotate = 50.21] [color={rgb, 255:red, 0; green, 0; blue, 0 }  ][line width=0.75]    (10.93,-3.29) .. controls (6.95,-1.4) and (3.31,-0.3) .. (0,0) .. controls (3.31,0.3) and (6.95,1.4) .. (10.93,3.29)   ;

\draw (576.93,162.51) node  [align=left] {. . .};
\draw (144,226) node  [align=left] {$\displaystyle ( x_{0} ,x_{1})$};
\draw (249,226) node  [align=left] {$\displaystyle ( x_{1} ,x_{2})$};
\draw (346,226) node  [align=left] {$\displaystyle ( x_{2} ,x_{3})$};
\draw (441,226) node  [align=left] {$\displaystyle ( x_{3} ,x_{4})$};
\draw (148,90) node  [align=left] {$\displaystyle ( x_{1} ,x_{0})$};
\draw (246,90) node  [align=left] {$\displaystyle ( x_{2} ,x_{1})$};
\draw (345,90) node  [align=left] {$\displaystyle ( x_{3} ,x_{2})$};
\draw (440,90) node  [align=left] {$\displaystyle ( x_{4} ,x_{3})$};
\draw (207,199) node  [align=left] {$\displaystyle \tfrac{1}{q}$};
\draw (300,199) node  [align=left] {$\displaystyle \tfrac{1}{q}$};
\draw (390,199) node  [align=left] {$\displaystyle \tfrac{1}{q}$};
\draw (490,199) node  [align=left] {$\displaystyle \tfrac{1}{q}$};
\draw (200,119) node  [align=left] {$\displaystyle 1$};
\draw (296,119) node  [align=left] {$\displaystyle 1$};
\draw (397,119) node  [align=left] {$\displaystyle 1$};
\draw (492,119) node  [align=left] {$\displaystyle 1$};
\draw (264,160) node  [align=left] {$\displaystyle \tfrac{q-1}{q}$};
\draw (359,160) node  [align=left] {$\displaystyle \tfrac{q-1}{q}$};
\draw (458,160) node  [align=left] {$\displaystyle \tfrac{q-1}{q}$};
\draw (185,160) node  [align=left] {$\displaystyle \tfrac{q-1}{q}$};
\draw (120,156) node  [align=left] {$\displaystyle 1$};

\end{tikzpicture}

\caption{Transition probabilities of $M_n$ when $\Gamma$ is a lattice of Nagao type (for the labeling of edges, see Fig.~\ref{fig:nagaolattice}).}
    \label{fig:nagaochain}
\end{figure}
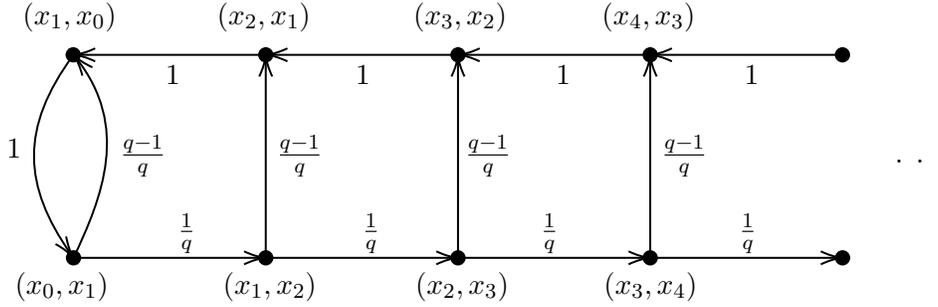

Consider a random trajectory of this Markov chain on its state space as depicted in the previous figure. The key phenomenon for us in this example is that once the trajectory turns toward the finite part (here, this corresponds to the edges facing left or up), it must deterministically walk all the way toward the finite part without a chance to turn around. This feature entails very strong recurrence properties which will allow us to control hitting times and, eventually, deduce  convergence of the Markov chain to the stationary measure (up to issues of periodicity) even with moving starting point. The latter property is crucial for Theorem \ref{thm.dm.geofin}. 
\end{example}

\subsubsection{Basic properties}

\begin{lemma}\label{lem:MCirred}
Let $\Gamma$ be a tree lattice. The associated Markov chain $M_n$ is irreducible. 
\end{lemma}

\begin{proof}
Since the graph $Q$ is connected, it is  sufficient to show that for any two edges $e,f\in EQ$, such that $\partial_1 e=\partial_0 f$, we have $P^n(e,f)>0$ for some $n\geq 1$. If $f\neq \overline{e}$, this holds for $n=1$ by definition of $P$. 

When $f= \overline{e}$, we claim that there exists $n\geq 0$ and a finite non-backtracking path of edges $(e=e_0,e_1, \ldots, e_n)$ (i.e. $\partial_1 e_i = \partial_0 e_{i+1}$ and $\partial_0 e_i \neq \partial_1 e_{i+1}$ for all $0\leq i<n$) such that $\ind(\overline{e}_n)>1$. Note that for all $i$, the transition probabilities $P(e_i,e_{i+1}), P(\overline{e_{i+1}},\overline{e_i}), P(e_n,\overline{e}_n)$ are all positive, implying $P_{2n+1}(e,\overline{e})>0$.

We show the existence of a path as above by contradiction. Suppose for any non-backtracking finite path starting at $e$ we have $\ind(\overline{e}_n)=1$. Such a path cannot end at a leaf, since then $\ind(\overline{e_n})=\deg(\partial_0 \overline{e_n})=\deg(\partial_1 e_n)>2$ by \eqref{eq:valency-gen}.  Hence, we can extend it to produce an infinite non-backtracking path with $\ind(\overline{e_i})=1$ for all $i\in \bbN$. In particular, $N_{\partial_0(e)}(e_i) \leq 1$ for all $i$, which contradicts the finiteness of the volume in \eqref{eq:volume-formula}.
\end{proof}

In the case of  geometrically finite lattices, we will prove positive recurrence of the associated Markov chain $M_n$ using Foster's drift criterion. Positive recurrence of $M_n$ in the setting of general tree lattices, which is required in the proof of Theorems \ref{thm.escape} and \ref{thm.dist.spheres}, is shown in Proposition \ref{prop.pos.rec} with a slightly more elaborate proof.
 
Assume $\Gamma$ is a geometrically finite tree lattice, $(Q,\ind)$ its associated edge-index graph and $F$ the finite part of $Q$.  For $e\in EQ$, we use the notation $|e|:=d(\partial_1(e),F)$ to indicate the distance between an edge and the finite part $F$. For $e\notin F$, we say that $e$ is \textit{oriented toward the finite part} if $d(\partial_1(e), F) < d(\partial_0(e), F)$, and \textit{oriented toward the cusp} otherwise.

\begin{lemma}\label{lem:recur-geomfin}
Let $\Gamma$ be a geometrically finite lattice. Then, the associated Markov chain $M_n$ is positive recurrent.
\end{lemma}

\begin{proof}
For $d_1,d_2 \geq 4$, one easily verifies that for any $e\in EQ$, setting
\[  V(e) = (3/2)^{|e|}\]
and letting $P$ to be the Markov operator corresponding to $M_n$, we have $PV(e) <\infty $  for all $e\in F$ and
\[  PV(e) \leq V(e) -1/8 , \quad \text{ for all } e\in EQ \setminus F. \]

In the case $d_1=d_2=3$, a slightly different function $V$ (which also works for the previous case) does the job:

Let 

\[  V(e)=\begin{cases}
0 & \text{if } e\in F, \\
 |e| & \text{if } e \text{ is oriented toward the finite part,} \\
 100(3/2)^{|e|} & \text{otherwise}.
\end{cases} \]
We have 
\[  PV(e) \leq V(e) -1 , \quad \text{ for all } e\in EQ \setminus F. \]

By Lemma \ref{lem:MCirred}, $M_n$ is irreducible. Hence, by Foster's drift Criteria \cite[Chapter 13]{meyn-tweedie}, $M_n$ is positive recurrent.
\end{proof}

A simple combinatorial observation allows us to show that when $\Gamma$ is geometrically finite, the period of the Markov chain $M_n$ is two. This is expressed in the following lemma:

\begin{lemma}\label{lemma.period}
If there exist two edges $e,f\in EQ$ such that $e\neq \overline{f}$, $\partial_1(f)=\partial_0(e)$ and $\ind(e),\ind(f)>1$, then the period of $M_n$ is $2$.
\end{lemma}

When $\Gamma$ is geometrically finite, one can simply take $e,f$ to be two consecutive edges in a Nagao ray oriented toward the finite part so that the lemma applies. 

\begin{proof}
Let $m-1$ be the length of a path from $e$ to $\bar{e}$ along edges with positive transition probabilities. Since $\ind(e)>1$, $P(\overline{e},e)>0$, hence there is a loop of length $m$ with positive transition probabilities along all edges. 
On the other hand, after the previous loop, one can follow the path from $e$ to $\overline{e}$, continue to $\overline{f}$, then to $f$ and finally back to $e$. This is a loop of length $m+2$. Hence, the period  divides by $m$ and $m+2$, which forces it to be $1$ or $2$. On the other hand, since $\Gamma$ action on $T$ preserves a partition into two sets of vertices (thanks to the assumption that $\Aut(T)$ acts without edge inversion), hence the period cannot be $1$, proving the claim.
\end{proof}

\subsubsection{Hitting time of the finite part}\label{subsub:returntime}
Let $F$ be the finite part of the graph $Q$. For the Markov chain $M_n$, we denote by $\tau$ the first hitting time of $F$ i.e.\ $\tau=\min \{n\in \mathbb{N} \, | \, \partial_1(M_n) \in F \}$. By positive recurrence, $\tau$ is finite almost surely. To deal with periodicity, define $\tau':= \min \{ n\in \bbN ~|~ n \geq \tau, 2|n \}$.



We start by a lemma that controls the probabilities of long hitting times of the finite part.

\begin{lemma}\label{lem:first_return}
Assume $\Gamma$ is geometrically finite. Let $\tau$ be the hitting time defined above. Then for any $e\in EQ$

\[
 \mathbb{P}_e ( \tau = i ) \leq  \begin{cases}
     q^{\lceil \frac{|e|-i}{2} \rceil}& \text{if } i\geq |e|,\\
    0              & \text{otherwise,}
\end{cases}
\]     
where $q+1=\min \{d_1 ,d_2\} $.
The same bound is true for $\tau'$ up to multiplicative constant $C=C(d_1,d_2)$.
\end{lemma}

\begin{proof}
Clearly, random walk starting at $e$ can never hit $F$ in less than $|e|$ steps. Similarly when $|e|=0$, the claim is obvious. 
When $|e|>0$, by definition of geometric finiteness,  $\partial_1(e)$ belongs to some Nagao ray. Because of the structure of Nagao rays (see Example \ref{ex:chain.on.nagao}), a Markov trajectory starting at an edge oriented toward the finite part $F$ must necessarily take at least one step toward $F$. This can only change once the trajectory visits $F$. Hence, if $e$ is oriented toward finite part, we deduce that $\mathbb{P}_e(\tau=|e|)=1$ matching the upper bound in the statement.

On the other hand, if $e$ is oriented toward the cusp, in order to avoid visiting $F$ in the first $i-1$ steps, the walk must take at least $\lfloor \frac{i- \vert e\vert}{2} \rfloor $ steps toward the cusp, all with probability $q^{-1}$. This gives the bound in the lemma.
\end{proof}

\subsubsection{Convergence of the Markov chain with varying initial distribution}
As before, let $P$ be the Markov operator corresponding to $M_n$. Let $\Omega_{0},\Omega_{1} \subseteq EQ$ be its cyclic classes and for $j=0,1$, denote by $\mu_{\Omega_j}$ the unique $P^2$ stationary probability measure on $\Omega_j$.

The next lemma describes the convergence of the Markov chain with moving initial distributions. The condition on the initial distributions will be clear later on, as this convergence will play a crucial role in the proof of Theorem \ref{thm.dm.geofin}.

\begin{lemma}\label{lemma:convergence}
Let $\Omega$ be a cyclic class of $P$ and $e(t) \subset \Omega$ be a sequence of edges in the same cyclic class, such that $t-|e(t)| \to \infty$. Let $n(t)$ be such that $|t-2n(t)|$ is constant so that $\delta_{e(t)}P^{2n(t)}$ is supported in $\Omega$. Then,
\[  \| \delta_{e(t)} P^{2n(t)} - \mu_\Omega \| \longrightarrow 0, \]
where $\|\cdot \|$ denotes the total variation norm (see e.g.\ \cite[\S D.1.2]{meyn-tweedie}).
\end{lemma} 

In the proof, we control the distributions with non-constant starting points $e(t)$ by studying the behaviour of the Markov chain conditioned on the hitting time of the finite part. This, together with the precise control on the hitting time as provided by Lemma \ref{lem:first_return}, allows us to prove the required convergence.

\begin{proof}
By conditioning the Markov chain on the hitting time $\tau'$ (as defined in \S  \ref{subsub:returntime}), we have
\begin{equation}\label{eq:condition_hitting_times}
    \|\delta_{e(t)}P^{2n(t)}-\mu_\Omega \|\leq 
    \sum_{\overset{i=0}{2|i}}^\infty  \mathbb{P}_{e(t)}(\tau'=i) \| \mathbb{P}_{e(t)}(\delta_{e(t)}P^{2n(t)} \in \cdot | \tau'=i)-\mu_\Omega\|.
\end{equation}
Here, for every $i \in 2\mathbb{N}$ with $\mathbb{P}_{e(t)}(\tau'=i)>0$, $\mathbb{P}_{e(t)}(\delta_{e(t)}P^{2n(t)} \in \cdot | \tau'=i)$ denotes the probability measure on $EQ$ given by 
$$e \mapsto \frac{\mathbb{P}_{e(t)}(M_{2n(t)}=e \, \, \text{and} \, \,  \tau'=i)}{\mathbb{P}_{e(t)}(\tau'=i)}.$$

It follows from strong Markov property that for $2n(t)\geq i$, we have 
\begin{equation}\label{eq:strongmarkov}
\| \mathbb{P}_{e(t)}(\delta_{e(t)}P^{2n(t)} \in \cdot | \tau'=i)- \mu_\Omega \| \leq \max_{e\in B(F,2)} \| \delta_e P^{2n(t)-i} - \mu_\Omega  \|,  
\end{equation}
where $B(F,2)$ is the set of all edges at distance $\leq 2 $ from $F$.

For $i \in 2\mathbb{N}$, denote
\[  A_i := \max_{e\in B(F,2)} \| \delta_e P^{i} - \mu_\Omega  \|  \qquad \text{and} \qquad  B_i^{(t)}:= \mathbb{P}_{e(t)}(\tau'=i)\]
For odd $i \in \mathbb{N}$, we set $A_i=B_i=0$. 
We have the following
\begin{enumerate}
    \item $\| m_1-m_2\|\leq 2$ for any probability measures $m_1,m_2$ on $EQ$,
    \item For each $t$, $\sum_i B_i^{(t)}=1$. 
    \item $B_i^{(t)} = 0$ for $i< |e(t)|-2$.
    \item $B_i^{(t)} \leq Cq^{\frac{|e(t)|-i}{2}}$ for $i\geq |e(t)|$.
    \item $A_i\to 0$ as $i\to \infty$.
\end{enumerate}
Indeed, (1) , (2) are trivial, and (3), (4) are proved in Lemma \ref{lem:first_return}. (5) holds since $P^2$ is positively recurrent (Proposition \ref{prop.pos.rec}), irreducible (Lemma \ref{lem:MCirred}) and aperiodic (lemma \ref{lemma.period}).

With this notation, splitting the right-hand-side of \eqref{eq:condition_hitting_times} into three sums, we get that left-hand-side of \eqref{eq:condition_hitting_times} is bounded above by
\begin{equation}\label{eq:split}
\sum_{i< |e(t)|-2}  A_{2n(t)-i} B_i^{(t)}  + \sum_{|e(t)|-2 \leq i\leq 2n(t)}  A_{2n(t)-i} B_i^{(t)}+ \sum_{i>2n(t)}  2 B_i^{(t)},
\end{equation}
where we used \eqref{eq:strongmarkov} for the first two sums, and (1) for the third. We need to show that the above tends to $0$ as $t\to \infty$. 

By (3), the first sum is identically $0$ and as $t \to \infty$, the third sum tends to $0$ by (4) and the fact that $2n(t)-|e(t)|$ tends to $\infty$.

We focus on the middle sum of \eqref{eq:split}, which after denoting $N_t=2n(t)-|e(t)|$, we rewrite  as follows
\[  \sum_{i=|e(t)|-2}^{2n(t)}  A_{2n(t)-i} B_i^{(t)} =  \sum_{i=0}^{2n(t)-|e(t)|+2} A_i B_{2n(t)-i}^{(t)} =    \sum_{i=0}^{N_t+2} A_i B_{N_t+|e(t)|-i}^{(t)} \]

\begin{equation}\label{eq:split2}
\leq \sum_{i=0}^{\lceil N_t/2 \rceil} A_i B_{N_t+|e(t)|-i}^{(t)}
+
\sum_{i= \lceil N_t/2 \rceil }^{N_t+2} A_i B_{N_t+|e(t)|-i}^{(t)}    
\end{equation}

By (2), the second sum in \eqref{eq:split2} is bounded from above by $ \displaystyle \sup_{i \geq N_t/2} \{ A_i\}$. As $t\to \infty$, $N_t$ goes to $\infty$ so that using (5), $\displaystyle \sup_{i \geq N_t/2} \{ A_i\}$ converges to $0$ showing that the second sum in \eqref{eq:split2} converges to $0$.

On the other hand, by $(1)$, the first sum in \eqref{eq:split2} is bounded above by

\[ 2 \sum_{i=0}^{\lceil N_t/2 \rceil}  B_{N_t+|e(t)|-i}^{(t)} \leq 
2 \sum_{i=\lceil N_t/2 \rceil +|e(t)|}^{N_t+|e(t)|} B_i^{(t)} .
\]
which converges to $0$ as $t\to \infty$ by (4). This concludes the proof.
\end{proof}

\subsection{Proof of Theorem \ref{thm.dm.geofin}}

We now link the Markov chain to the study of orbital measures of horospherical orbits and use the properties of $M_n$ to prove Theorem \ref{thm.dm.geofin}. Before starting the proof, we remark that it suffices to prove the result only for the F\o lner sequence $F_t$. Indeed, let $O$ be a $M$-invariant compact subset with non-empty interior in $G^0_\eta$, $a \in G$ be a hyperbolic element with attractive fixed point $\eta$ and of (minimal) translation distance 2 and $O_{t}=a^t O a^{-t}$ be the associated good F\o lner sequence. It follows by compactness of $F_0$ and $O$ that for some $n_0 \in \mathbb{N}$ and every $t \in \mathbb{N}$, we have
\begin{equation}\label{eq.folners}
O_{t-n_0}\subseteq F_{2t} \subseteq O_{t+n_0}.    
\end{equation}
As a consequence, there exists $c \in (0,1)$ such that for every $t \in \mathbb{N}$, the sequence $F_{2t}=a^t F_0 a^{-t}$ satisfies \begin{equation}\label{eq.folners1}
0<c \leq \frac{m_{G^0_\eta}(F_{2t})}{m_{G^0_\eta}(O_{t+n_0)}}
 \leq \frac{m_{G^0_\eta}(F_{2t})}{m_{G^0_\eta}(O_{t-n_0)}} \leq \frac{1}{c} < \infty  
\end{equation}
One easily sees from these inequalities that the orbital measures $\nu_{x,t}$ associated to $F_{t}$ have non-escape of mass if and only if those associated to $O_t$ have it.

\subsubsection{Reduction to measures on the tree}\label{subsub:reduction}

For the rest of the section we fix $x = g\Gamma \in X$ with non-compact $G^0_\eta$-orbit. 

Recall that for $t \in \mathbb{N}$,  $\nu_{x,t}$ denotes the probability measure on the orbit $F_t x$ obtained by pushforward of the Haar probability measure on $F_t$ under the orbit map $u\mapsto ux$ for $u\in F_t$.  

Denote by $\sigma_t$ the uniform probability measure on the finite set $g^{-1}F_t \tilde{o} \subset VT$. 
The following observation is the first step in reducing the proof of recurrence of horospherical orbits to  studying recurrence properties of the Markov chain $M_n$ introduced earlier.

\begin{lemma}\label{lem:nu_and_sigma*}
For every $t \in \mathbb{N}^\ast$, we have
\begin{equation}\label{eq:nu_and_sigma*}
    \proj_* \nu_{x,t} = \pi_* \sigma_t .
\end{equation}  
\end{lemma}

\begin{proof}
Recall that $x\in G/\Gamma$ is fixed and  $g\in G$ is such that $x=g\Gamma$. Consider the map $f: G^0_\eta \to T$ given by $f(u)=g^{-1}u^{-1}\tilde{o}$. Denote by $O: u \mapsto ug\Gamma$ the orbit map. Then the following diagram clearly commutes:
\begin{figure}[H]
    \centering
\begin{tikzpicture}
  \matrix (m) [matrix of math nodes,row sep=3em,column sep=4em,minimum width=2em]
  {
     F_t  & g^{-1}F_t \tilde{o} \subset VT \\
     F_t g\Gamma \subset G/\Gamma
     & \Gamma g^{-1}F_t\tilde{o} \subset VQ \\};
  \path[-stealth]
  
    (m-1-1) 
             edge node [above] {$f$} (m-1-2)
             edge node [left] {$O$} (m-2-1)

     (m-1-2) edge node [right] {$\pi$} (m-2-2)
     (m-2-1) edge node [below] {$\proj$}    (m-2-2) ;
\end{tikzpicture}
\end{figure}
By definition, $O_* m_{F_t}=\nu_{x,t}$ and hence it is enough to see that $f_* m_{F_t}=\sigma_t$. This is readily verified and we are done. 
\end{proof}

\subsubsection{Further reduction to shadows and the Markov chain}

Above, we related the orbital measures $\nu_{x,t}$ to $\sigma_t$  - distributions on $VT$. The next lemmas will link $\sigma_t$ to the distributions of the Markov chain. 

For $v \in VT$ and $n \in \mathbb{N}$, denote by $S(v,n)$ the set of vertices of $T$ at distance $n \geq 0$ from $v$. For $w$ a neighbor of $v$, let $S_w(v,n)$ be the subset of $S(v,n)$ consisting of vertices $z \in VT$ such that $d(z,w)<d(z,v)$. Thinking of $v$ as a light source at the center the sphere, we call $S_w(v,n)$ the \textit{shadow} of $w$ (see Fig.~\ref{fig:shadow} for illustration). Denote by $\lambda_{(v,w),n}$ the uniform probability measure on the shadow $S_w(v,n)$.

\begin{figure}[H]
    \centering
    
\tikzset{every picture/.style={line width=0.75pt}} 

\begin{tikzpicture}[x=0.75pt,y=0.75pt,yscale=-1,xscale=1]

\draw  [dash pattern={on 0.84pt off 2.51pt}]  (178.5,200) -- (230.5,200) ;

\draw  [dash pattern={on 0.84pt off 2.51pt}]  (230.5,200) -- (261.5,161) ;

\draw  [dash pattern={on 0.84pt off 2.51pt}]  (230.5,200) -- (260.5,239) ;

\draw  [dash pattern={on 0.84pt off 2.51pt}]  (261.49,160.47) -- (309.74,141.09) ;
\draw [shift={(309.74,141.09)}, rotate = 338.11] [color={rgb, 255:red, 0; green, 0; blue, 0 }  ][fill={rgb, 255:red, 0; green, 0; blue, 0 }  ][line width=0.75]      (0, 0) circle [x radius= 3.35, y radius= 3.35]   ;

\draw  [dash pattern={on 0.84pt off 2.51pt}]  (307.5,177) -- (261.49,160.47) ;

\draw [shift={(307.5,177)}, rotate = 199.76] [color={rgb, 255:red, 0; green, 0; blue, 0 }  ][fill={rgb, 255:red, 0; green, 0; blue, 0 }  ][line width=0.75]      (0, 0) circle [x radius= 3.35, y radius= 3.35]   ;
\draw  [dash pattern={on 0.84pt off 2.51pt}]  (260.74,239.56) -- (309.93,256.4) ;
\draw [shift={(309.93,256.4)}, rotate = 18.89] [color={rgb, 255:red, 0; green, 0; blue, 0 }  ][fill={rgb, 255:red, 0; green, 0; blue, 0 }  ][line width=0.75]      (0, 0) circle [x radius= 3.35, y radius= 3.35]   ;

\draw  [dash pattern={on 0.84pt off 2.51pt}]  (260.5,239) -- (308.5,215) ;
\draw [shift={(308.5,215)}, rotate = 333.43] [color={rgb, 255:red, 0; green, 0; blue, 0 }  ][fill={rgb, 255:red, 0; green, 0; blue, 0 }  ][line width=0.75]      (0, 0) circle [x radius= 3.35, y radius= 3.35]   ;

\draw  [dash pattern={on 0.84pt off 2.51pt}]  (178.5,200) -- (147.75,239.2) ;

\draw  [dash pattern={on 0.84pt off 2.51pt}]  (177.93,199.4) -- (147.67,160.59) ;

\draw  [dash pattern={on 0.84pt off 2.51pt}]  (147.19,239.12) -- (103,271) ;

\draw  [dash pattern={on 0.84pt off 2.51pt}]  (103,233) -- (147.19,239.12) ;

\draw  [dash pattern={on 0.84pt off 2.51pt}]  (147.44,160.03) -- (103,140) ;

\draw  [dash pattern={on 0.84pt off 2.51pt}]  (147.67,160.59) -- (103,183) ;

\draw  [dash pattern={on 0.84pt off 2.51pt}]  (103,140.03) -- (55,124.66) ;
\draw [shift={(52,124)}, rotate = 196.26] [color={rgb, 255:red, 0; green, 0; blue, 0 }  ][line width=0.75]      (0, 0) circle [x radius= 3.35, y radius= 3.35]   ;

\draw  [dash pattern={on 0.84pt off 2.51pt}]  (103,140.59) -- (55,151.52) ;
\draw [shift={(52,152)}, rotate = 168.11] [color={rgb, 255:red, 0; green, 0; blue, 0 }  ][line width=0.75]      (0, 0) circle [x radius= 3.35, y radius= 3.35]   ;

\draw  [dash pattern={on 0.84pt off 2.51pt}]  (103,183) -- (55,167.23) ;
\draw [shift={(52,166.48)}, rotate = 198.52] [color={rgb, 255:red, 0; green, 0; blue, 0 }  ][line width=0.75]      (0, 0) circle [x radius= 3.35, y radius= 3.35]   ;

\draw  [dash pattern={on 0.84pt off 2.51pt}]  (103,183) -- (55,199.2) ;
\draw [shift={(52,200)}, rotate = 160.11] [color={rgb, 255:red, 0; green, 0; blue, 0 }  ][line width=0.75]      (0, 0) circle [x radius= 3.35, y radius= 3.35]   ;

\draw  [dash pattern={on 0.84pt off 2.51pt}]  (55,220.68) -- (103,233) ;

\draw [shift={(52,220)}, rotate = 16.82] [color={rgb, 255:red, 0; green, 0; blue, 0 }  ][line width=0.75]      (0, 0) circle [x radius= 3.35, y radius= 3.35]   ;
\draw  [dash pattern={on 0.84pt off 2.51pt}]  (55,264.3) -- (103,271) ;

\draw [shift={(52,264)}, rotate = 7.39] [color={rgb, 255:red, 0; green, 0; blue, 0 }  ][line width=0.75]      (0, 0) circle [x radius= 3.35, y radius= 3.35]   ;
\draw  [dash pattern={on 0.84pt off 2.51pt}]  (103,233) -- (55,248.2) ;
\draw [shift={(52,249)}, rotate = 160.02] [color={rgb, 255:red, 0; green, 0; blue, 0 }  ][line width=0.75]      (0, 0) circle [x radius= 3.35, y radius= 3.35]   ;

\draw  [dash pattern={on 0.84pt off 2.51pt}]  (103,271) -- (55,289.22) ;
\draw [shift={(52,290)}, rotate = 160.62] [color={rgb, 255:red, 0; green, 0; blue, 0 }  ][line width=0.75]      (0, 0) circle [x radius= 3.35, y radius= 3.35]   ;

\draw (183,188) node  [align=left] {$\displaystyle v$};
\draw (226,188) node  [align=left] {$\displaystyle w$};

\end{tikzpicture}

\begin{center}

\end{center}
\caption{Shadow of $w$: all nodes are vertices on the sphere $S(v,3)$, the shadow $S_w(v,3)$ consists of solid nodes.}
    \label{fig:shadow}
\end{figure}
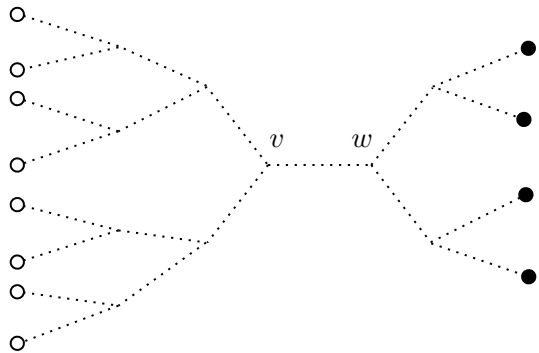


%
%
%


\begin{lemma}\label{lemma.side.sphere}
Let $G$ be a non-compact closed subgroup of $\Aut(T)$ that acts transitively on $\partial T$. For any $t \in \mathbb{N}^\ast$, we have

\begin{equation*}
     \sigma_t= \frac{1}{\deg(g^{-1}y_t)-1} \sum_{i_t=1}^{\deg(g^{-1}y_t)-1}  \lambda_{(g^{-1}y_t,\tilde{v}_{i_t}),t}=\lambda_{(g^{-1}y_{t+1},g^{-1}y_t),t+1} ,
\end{equation*}
where  $\{\tilde{v}_{i_t}\}$ is the collection of vertices in $T$ neighboring $g^{-1}y_t$ except $g^{-1}y_{t+1}$.
\end{lemma}

\begin{proof}
The last equality directly results from the definition of the probability measure $\lambda_{(v,w),t}$, therefore we focus on the first equality. Since all the shadows involved have the same cardinality, using the definitions of $\sigma_t$ and $\lambda_{(v,w),t}$, the equality will follow if we show 
\begin{equation*}
g^{-1}F_t \tilde{o} =     \bigsqcup_{i_t=1}^{\deg(g^{-1}y_t)-1} S_{\tilde{v}_{i_t}}(g^{-1}y_t,t).
\end{equation*}
In other words, the set $g^{-1}F_t\tilde{o}$ is the set of vertices on the sphere of radius $t$ around $g^{-1}y_t$ except the shadow of $g^{-1}y_{t+1}$.

Since $g$ acts by isometry, it is enough to show 
this for $g=\id$.
We clearly have 
$$F_t \tilde{o} \quad  \subset \bigsqcup_{\substack{(y_t,w)\in ET \\ w\neq y_{t+1}}} S_{w}(y_t,t).$$

To show the other inclusion, let $\xi_1,\xi_2 \in \partial T$ be such that $(\xi_i, \eta) \cap [y_0, \eta) \supset [y_t, \eta)$ for $i=1,2$. It clearly suffices to show that there exist a sequence $h_n \in F_t$ with $h_n \xi_1 \to \xi_2$ as $n \to \infty$. To see this, note that since $G$ is non-compact, closed and transitive on $\partial T$, by \cite[Lemma 3.1.1]{burger-mozes.local} it acts doubly transitively on $\partial T$.
Furthermore, since it is non-compact, it contains a hyperbolic element $a$ that -thanks to double transitivity- we can suppose to have attracting point $\eta$ and repelling point $\xi_1$ on $\partial T$. Similarly up to conjugating $a$, let $b$ be a hyperbolic element with attracting fixed point $\eta$ and repelling fixed point $\xi_2$. The sequence $h_n=b^{-n}a^n$ does the job and this concludes the proof.
\end{proof}

Let
\begin{equation}\label{eq:def_D}
    D:= \frac{1}{\deg(\tilde{o})} \sum_{(\tilde{o},\tilde{w})\in ET} \delta_{(o,\pi(\tilde{w}))}.
\end{equation}  
where deg$(.)$ denotes the valency of the vertex $\tilde{o}$. Denote by $\rho_n$ the uniform measure on the sphere $S(\tilde{o},n)$. In the following lemma, we realize the probability measures $\pi_\ast \lambda_{(v,w),n}$ and $\pi_* \rho_n$ as the $n^{th}$-step distribution of our Markov chain with appropriate initial distributions. The fact that such a relation exists is not surprising as the Markov chain $M_n$ is obtained as a quotient of the simple random walk on the edges of the tree $T$.

\begin{lemma}\label{lemma.mc.spheres}
Let $\tilde{v} \in VT$ and $\tilde{w} \in VT$ be a neighbour of $\tilde{v}$. Denote $\pi(\tilde{v})=v$, $\pi(\tilde{w})=w$. Then for any $n\geq 0$
\begin{equation}\label{eq:mc.shadows}
    \pi_\ast \lambda_{(\tilde{v},\tilde{w}),n+1} = \partial_1{}_\ast (\delta_{(v,w)}P^n).
\end{equation}  
and 
\begin{equation}\label{eq:sphere-sum_of_pacman}
     \pi_* \rho_{n+1} = \partial_{1*} (D P^n).
\end{equation} \end{lemma}

\begin{proof}
Define a Markov chain $L_n$ with state space $ET$ and transition probabilities
\[  l(e_1,e_2)=
\begin{cases}
\frac{1}{\deg(\partial_1(e_1)-1)} & \text{if } \partial_1(e_1)=\partial_0(e_2),\\
0 & \text{otherwise.}
\end{cases}
\]
This chain describes the unbiased non-backtracking random walk on the (directed) edges of the tree.  For all $n\geq 0$, one clearly has
\[   \lambda_{(\tilde{v},\tilde{w}),n+1}  = \partial_{1*} (\delta_{(\tilde{v},\tilde{w})} L_n). \]

On the other hand, observe that $\pi_* \delta_{(\tilde{v},\tilde{w})} L_n = \delta_{(v,w)}M_n$ but since $\pi$ commutes with $\partial_1$, \eqref{eq:mc.shadows} follows.

To see the second claim, note that by construction of the Markov chain $L_n$, we have
$$\rho_{n+1}= \frac{1}{\deg(\tilde{o})}\sum_{(\tilde{o},\tilde{w})\in ET} \delta_{(\tilde{o},\tilde{w})}L_n. $$
Applying $\pi_*$ to both sides yields \eqref{eq:sphere-sum_of_pacman}.
\end{proof}

\begin{remark}\label{rk.general}
We remark here that the statements of Lemmas \ref{lem:nu_and_sigma*}, \ref{lemma.side.sphere} and \ref{lemma.mc.spheres} hold more generally for any lattice $\Gamma$ of $\Aut(T)$. Indeed, the proofs do not make use of the particular structure of a geometrically finite lattice. 
\end{remark}

\begin{proposition}\label{more.precise}
Let  $x=g\Gamma$ with non-compact $G_\eta^0$-orbit. Then the set of weak-$*$ limit points of $  \pi_* \sigma_t$ is  \{$\partial_{1*} \mu_{\Omega_0},\partial_{1*} \mu_{\Omega_1}\}$, where $\Omega_i$'s are two cyclic classes of $P$ and for $i=0,1$, $\mu_{\Omega_i}$ is the unique $P^2$-stationary measures on $\Omega_i$ as before.
\end{proposition}

\begin{remark}\label{rk.limit.independent}
In this proposition, the measures $\pi_\ast \sigma_t$ depend on the point $x=g\Gamma$, but the set of limit points of $\pi_\ast \sigma_t$ does not.
\end{remark}

\begin{proof}




Combining Lemmas  \ref{lemma.side.sphere} and \ref{lemma.mc.spheres} and denoting $v_{i_t}:=\pi(\tilde{v}_{i_t})$, we have for any $t \in \mathbb{N}^\ast$
\begin{equation}\label{eq5}
\pi_* \sigma_t 
=\frac{1}{\deg(g^{-1}y_t)-1}\sum_{i_t=1}^{\deg(g^{-1}y_t)-1} \partial_1 {}_\ast (\delta_{(\pi(g^{-1}y_t),v_{i_t})}P^{t-1}).
\end{equation}

For a fixed $t \in \mathbb{N}$, the edges $(\pi(g^{-1}y_t),v_{i_t})$ belong to the same cyclic class, denote it by $\Omega_{j(t)}$. Up to passing to a subsequence (i.e.\ considering even or odd $t$'s), which we also denote by $t$, we may assume that $j(t)$ is constant. For each $t$, choose one vertex $v(t) \subset \{ v_{i_t} \}$ and denote the edge 
\begin{equation}\label{eq.defn.et}
e(t)= (\pi(g^{-1}y_t),v_t).
\end{equation}




Up to passing to a further subsequence of $t$'s, we may suppose that $\delta_{e(t)}P^{t-1}$ is supported in a single cyclic class.  Therefore, for some $r \in \{0,1\}$, every $t$ in this sequence writes as $t-1=2n(t)+r$, where $n(t) \in \mathbb{N}$. Thus we can write
\begin{equation}\label{eq7}
\delta_{e(t)}P^{t-1}=\delta_{e(t)}P^{2n(t)}P^r.
\end{equation}
Now, since by Proposition \ref{prop.irrational}, we have $t-|e(t)|\to \infty$, Lemma \ref{lemma:convergence} applies and we deduce that $\|\delta_{e(t)}P^{2n(t)} -\mu_{\Omega_j}\| \to 0$ as $t \to \infty$ for some $j \in \{0,1\}$. Therefore, we have
\[\|\delta_{e(t)}P^{t-1}-\mu_{\Omega_{i}}\| \to 0  \text{ as } t \to \infty. \]
where $i=j+r \; (\text{mod}\, 2)$. This finishes the proof. 
\end{proof}

\subsubsection{Proof of Theorem \ref{thm.dm.geofin}} 
With the notation of this section, we want to show that for any $\epsilon>0$, there exists a compact set $K\subset G/\Gamma$ such that for every $x\in X$ with non-compact $G^0_\eta$-orbit, there exists $N \in \mathbb{N}$ such that for all $t>N$
\begin{equation}\label{eq.thm.a}
\nu_{x,t}(K) \geq 1-\epsilon. 
\end{equation}

Let $L\subset EQ$ be a finite set such that $\mu_{\Omega_j}(L)>1-\varepsilon$ for $j\in \{0,1\}$. The set $K=\proj^{-1}(\partial_1(L))$ is compact, since the map $\proj: G/ \Gamma \to VQ$ has compact fibers. Using Lemma \ref{lem:nu_and_sigma*} and Proposition \ref{more.precise}, we have for any $x\in X$
     \[ \liminf_{t \to \infty} \nu_{x,t}(K) = \liminf_{t \to \infty} \pi_* \sigma_t(\partial_1(L))  \geq \min\{\mu_{\Omega_0}(L),\mu_{\Omega_1}(L)\}  > 1-\varepsilon,     \]
and the claim follows.




\section{Equidistribution}\label{sec.equidist}
This section is devoted to the proof of Theorem \ref{thm.equidist} which we deduce from Theorem \ref{thm.dm.geofin} and our previous work \cite{CFS}.

Fix a hyperbolic element $a \in G$ of translation length $2$ with attracting fixed point $\eta$. Denote by $\eta_- \in \partial T$ the repelling point of $a$ and set $M=G^0_\eta \cap G^0_{\eta_-}$.  Let $O$ be a $M$-invariant compact subset of $G_\eta^0$ with non-empty interior. Let $O_t=a^t O a^{-t}$ be the associated good F{\o}lner sequence for $G_\eta^0$. As before, for  $x \in X$, denote by $\nu_{x,t}$  the orbital measure $ m_{O_t}* \delta_x$.

Let $x\in X$ be such that $G^0_\eta$-orbit of $x$ is not compact. By Theorem \ref{thm.dm.geofin}, up to passing to a subsequence, we can suppose that
\begin{equation}\label{eq.m}
\nu_{x,t} \longrightarrow m 
\end{equation}
for the weak-$\ast$ topology and where $m$ is a Borel probability measure on $X$. Furthermore, since $O_t$ is a F{\o}lner sequence, $m$ is $G^0_\eta$-invariant. We need to show that $m=m_X$.

Recall that by \cite[Theorem 1.6]{CFS}, there exists countably many closed $G^0_\eta$-orbits in $X$. These are all compact and for each cusp of $\Gamma$, there exists precisely a discrete one parameter family of compact orbits. Denote by $k\in \mathbb{N}$ the number of cusps of $\Gamma$ and let $C_{i,j}$ be the collection of compact $G^0_\eta$-orbits, where $i=1,\ldots, k$ and $j \in \mathbb{Z}$. By the same result, we have $a C_{i,j}=C_{i,j+1}$ and $a^{-\ell}C_{i,j}$ escapes to infinity as $\ell \to \infty$ in the sense that for any compact set $K$, we have $K \cap a^{-\ell}C_{i,j}= \emptyset$ for every $\ell$ large enough (see e.g.\ proof of \cite[Lemma 6.2]{CFS}).

We first prove that $m(C_{i,j})=0$ for every $i=1,\ldots,k$ and $j \in \mathbb{Z}$. For a contradiction, suppose $m(C_{i_0,j_0})>0$ for some $i_0,j_0$. Denote $\frac{1}{2}m(C_{i_0,j_0})=:\epsilon>0$ and let $K=K(\epsilon)$ be the compact subset of $X$ given by Theorem \ref{thm.dm.geofin}. It follows by the latter result that we have
\begin{equation}\label{eq.choosing.K}
m(K) \geq 1-\epsilon.
\end{equation}

Choose an $\ell \in \mathbb{N}$ large enough so that $a^{-\ell}C_{i_0,j_0} \cap K =\emptyset$. Since $a^{-\ell}x$ does not lie on a compact $G^0_\eta$-orbit either, using Theorem \ref{thm.dm.geofin},  by passing to a further subsequence in \eqref{eq.m}, we can suppose that $\nu_{a^{-\ell}x,t}$ also converges to a $G^0_\eta$-invariant probability measure that we denote by $m^{a^{-\ell}}$. As in \eqref{eq.choosing.K}, by Theorem \ref{thm.dm.geofin}, we have $m^{a^{-\ell}}(K)\geq 1-\epsilon.$

Using the relation $a^{-\ell}O_t a^{\ell}=O_{t-\ell}$, one verifies by a simple calculation that we have $m^{a^{-\ell}}=a^{-\ell}_\ast m$. Using this, we deduce 
$$2\epsilon= m(C_{i_{0},j_0})=m^{a^{-\ell}}(a^{-\ell}C_{i_{0},j_0}) \leq m^{a^{-\ell}}(X \setminus K) \leq \frac{1}{2}m(C_{i_0,j_0})=\epsilon,$$
a contradiction.  Therefore, $m(C_{i,j})=0$ for all $i=1,\ldots,k$ and $j \in \mathbb{Z}$.

We mention that at this point, one could conclude the proof by appealing to the classification of ergodic $G_\eta^0$-invariant Borel probability measures \cite[Theorem 1.1]{CFS}. However, that result has extra hypotheses on $G$, namely Tits independence property and a certain transitivity condition. On the other hand, for a geometrically finite lattice $\Gamma$, it is possible to give a similar classification of ergodic $G_\eta^0$-invariant Borel probability measures on $G/\Gamma$ for a more general group $G$ as in Theorem \ref{thm.equidist}. We single this out in the next proposition which is essentially contained in \cite{CFS}.

\begin{proposition}\label{prop.geofin.measure}
Let $T$ be a $(d_1,d_2)$-biregular tree, with $d_1,d_2\geq 3$, and $G$ a non-compact, closed and topologically simple subgroup of $\Aut(T)$ acting transitively on $\partial T$. Let $\Gamma$ be a geometrically finite lattice in $G$ and $\eta \in \partial T$. Then, any $G^0_\eta$-invariant and ergodic Borel probability measure on $X=G/\Gamma$ is either $G^0_\eta$-homogeneous and compactly supported, or it is the Haar measure $m_X$.
\end{proposition}

To finish the proof of Theorem \ref{thm.equidist}, consider an ergodic decomposition of the $G^0_\eta$-invariant probability measure $m$. Since there are countably many closed $G^0_\eta$-orbits and each of them has zero measure with respect to $m$, the same holds for almost every ergodic component of $m$. Therefore by  Proposition \ref{prop.geofin.measure} almost every ergodic component of $m$ is the Haar measure $m_X$, hence $m=m_X$ completing the proof of Theorem \ref{thm.equidist}. \qed

\begin{proof}[Proof of Proposition \ref{prop.geofin.measure}]
We use the same notation introduced in the beginning of the proof of Theorem \ref{thm.equidist}, namely $a$ is a hyperbolic element with attracting fixed point $\eta \in \partial T$, the group $M$ and the good F\o lner sequence $O_t$ are as defined there. Let $m_0$ be a $G^0_\eta$-invariant and ergodic probability measure on $X$. If $m_0$ gives positive mass to a compact $G^0_\eta$-orbit, then by ergodicity, it must be the homogeneous measure supported on that orbit.
So let us suppose that $m_0$ gives zero mass to each compact $G^0_\eta$-orbit. By pointwise ergodic theorem for amenable groups (\cite[Theorem 1.2]{lindenstrauss}),  there exists a point $y \in X$ that is generic with respect to $m_0$ and the tempered F{\o}lner sequence $O_t$ (see \cite[\S 2.3]{CFS}). By \cite[Theorem 1.6]{CFS}, $G^0_\eta$-orbit of $y$ is dense in $X$. Then, by \cite[Lemma 6.2]{CFS}, there exists a compact set $K'$ in $X$, a sequence of integers $n_k \to \infty$ such that $a^{-n_k}y \in K'$ for every $k \in \mathbb{N}$. For a function $\theta \in C_c(X)$, denote  $\hat{\theta}(z):=\int \theta(mz)dm_M(z)$ where $m_M$ is the Haar probability measure on $M$. The function $\hat{\theta}$ is clearly $M$-invariant. Since $G$ is closed, transitive and topologically simple, it has the Howe--Moore property \cite[Proposition 4.2]{burger-mozes.product} and in particular the action of the hyperbolic element $a$ on $X$ is mixing. Therefore we can apply \cite[Lemma 6.3]{CFS}, where we can take $O^+$ to be $O_{t_0}$ for some $t_0 \in \mathbb{Z}$ small enough, for every $\theta \in C_c(X)$, we have

\begin{equation}\label{eq.howe.moore}
\frac{1}{m_{G^0_\eta}(O_{t_0+n_k})} \int_{O_{t_0+n_k}} \hat{\theta}(uy) dm_{G^0_\eta}(u) \underset{k \to \infty}{\longrightarrow} \int \hat{\theta}(z) dm_X(z).
\end{equation}

On the other hand by choice of $y \in X$, the left-hand-side above also converges to $\int \hat{\theta}(z) dm_0(z)$. It follows 

$$
\int \int \theta(mz) dm_M(m) dm_0(z)=\int \int \theta(mz) dm_M(m) dm_X(z).
$$

But since $m_0$ and $m_X$ are $G^0_\eta$-invariant, by Fubini's theorem, it follows that

$$
\int \theta(z) dm_0(z)=\int \theta(z)  dm_X(z)
$$
in other words, $m_0=m_X$ as required.
\end{proof}



\section{Escape of mass phenomenon}\label{sec.escape}

This section contains the proofs of Theorems \ref{thm.escape} and \ref{thm.escape.and.equidist}.

We start by proving an escape of mass result that implies Theorem \ref{thm.escape}. Regarding the construction of a lattice $\Gamma<\Aut(T)$ that figures in the following result, we note that by \cite[\S 4.11, Example 1]{bass-lubotzky}, for every $q \geq 2$, there exists a lattice $\Gamma \leq G = \Aut(T_{2q+2})$ whose associated edge-indexed graph is as in Fig.~\ref{fig:2q-ray}. Clearly, this $\Gamma$ is not geometrically finite. 

\begin{figure}[H]
    \centering

\tikzset{every picture/.style={line width=0.75pt}} 

\begin{tikzpicture}[x=0.75pt,y=0.75pt,yscale=-1,xscale=1]

\draw    (31.85,130.02) -- (128.15,130.98) ;
\draw [shift={(130.5,131)}, rotate = 0.57] [color={rgb, 255:red, 0; green, 0; blue, 0 }  ][line width=0.75]      (0, 0) circle [x radius= 3.35, y radius= 3.35]   ;
\draw [shift={(29.5,130)}, rotate = 0.57] [color={rgb, 255:red, 0; green, 0; blue, 0 }  ][line width=0.75]      (0, 0) circle [x radius= 3.35, y radius= 3.35]   ;
\draw    (131.75,128.65) .. controls (152.57,91.58) and (209.19,91.8) .. (230.55,130.22) ;
\draw [shift={(231.5,132)}, rotate = 62.88] [color={rgb, 255:red, 0; green, 0; blue, 0 }  ][line width=0.75]      (0, 0) circle [x radius= 3.35, y radius= 3.35]   ;
\draw [shift={(130.5,131)}, rotate = 296.57] [color={rgb, 255:red, 0; green, 0; blue, 0 }  ][line width=0.75]      (0, 0) circle [x radius= 3.35, y radius= 3.35]   ;
\draw    (131.69,133.29) .. controls (151.58,169.43) and (208.21,169.25) .. (230.51,133.65) ;
\draw [shift={(231.5,132)}, rotate = 300.07] [color={rgb, 255:red, 0; green, 0; blue, 0 }  ][line width=0.75]      (0, 0) circle [x radius= 3.35, y radius= 3.35]   ;
\draw [shift={(130.5,131)}, rotate = 64.03] [color={rgb, 255:red, 0; green, 0; blue, 0 }  ][line width=0.75]      (0, 0) circle [x radius= 3.35, y radius= 3.35]   ;
\draw    (232.75,129.65) .. controls (253.57,92.58) and (310.19,92.8) .. (331.55,131.22) ;
\draw [shift={(332.5,133)}, rotate = 62.88] [color={rgb, 255:red, 0; green, 0; blue, 0 }  ][line width=0.75]      (0, 0) circle [x radius= 3.35, y radius= 3.35]   ;
\draw [shift={(231.5,132)}, rotate = 296.57] [color={rgb, 255:red, 0; green, 0; blue, 0 }  ][line width=0.75]      (0, 0) circle [x radius= 3.35, y radius= 3.35]   ;
\draw    (232.69,134.29) .. controls (252.58,170.43) and (309.21,170.25) .. (331.51,134.65) ;
\draw [shift={(332.5,133)}, rotate = 300.07] [color={rgb, 255:red, 0; green, 0; blue, 0 }  ][line width=0.75]      (0, 0) circle [x radius= 3.35, y radius= 3.35]   ;
\draw [shift={(231.5,132)}, rotate = 64.03] [color={rgb, 255:red, 0; green, 0; blue, 0 }  ][line width=0.75]      (0, 0) circle [x radius= 3.35, y radius= 3.35]   ;
\draw    (333.75,130.65) .. controls (354.57,93.58) and (411.19,93.8) .. (432.55,132.22) ;
\draw [shift={(433.5,134)}, rotate = 62.88] [color={rgb, 255:red, 0; green, 0; blue, 0 }  ][line width=0.75]      (0, 0) circle [x radius= 3.35, y radius= 3.35]   ;
\draw [shift={(332.5,133)}, rotate = 296.57] [color={rgb, 255:red, 0; green, 0; blue, 0 }  ][line width=0.75]      (0, 0) circle [x radius= 3.35, y radius= 3.35]   ;
\draw    (333.69,135.29) .. controls (353.58,171.43) and (410.21,171.25) .. (432.51,135.65) ;
\draw [shift={(433.5,134)}, rotate = 300.07] [color={rgb, 255:red, 0; green, 0; blue, 0 }  ][line width=0.75]      (0, 0) circle [x radius= 3.35, y radius= 3.35]   ;
\draw [shift={(332.5,133)}, rotate = 64.03] [color={rgb, 255:red, 0; green, 0; blue, 0 }  ][line width=0.75]      (0, 0) circle [x radius= 3.35, y radius= 3.35]   ;

\draw (479.35,131) node [scale=1.25] [align=left] {. . .};
\draw (47,140) node [scale=0.7] [align=left] {$\displaystyle 2q+2$};

\draw (117,140) node [scale=0.7] [align=left] {$\displaystyle 2q$};
\draw (225.5,151) node [scale=0.7] [align=left] {$\displaystyle q$};
\draw (137,151) node [scale=0.7] [align=left] {$\displaystyle 1$};
\draw (137,110) node [scale=0.7] [align=left] {$\displaystyle 1$};
\draw (225.5,110) node [scale=0.7] [align=left] {$\displaystyle q$};
\draw (326.5,151) node [scale=0.7] [align=left] {$\displaystyle q$};
\draw (236.5,151) node [scale=0.7] [align=left] {$\displaystyle 1$};
\draw (236.5,110) node [scale=0.7] [align=left] {$\displaystyle 1$};
\draw (326.5,110) node [scale=0.7] [align=left] {$\displaystyle q$};
\draw (427.5,151) node [scale=0.7] [align=left] {$\displaystyle q$};
\draw (336.5,151) node [scale=0.7] [align=left] {$\displaystyle 1$};
\draw (336.5,110) node [scale=0.7] [align=left] {$\displaystyle 1$};
\draw (427.5,110) node [scale=0.7] [align=left] {$\displaystyle q$};
\draw (150,129) node  [align=left] {$\displaystyle x_{1}$};
\draw (252,129) node  [align=left] {$\displaystyle x_{2}$};
\draw (351,129) node  [align=left] {$\displaystyle x_{3}$};
\draw (448,129) node  [align=left] {$\displaystyle x_{4}$};
\draw (18,127) node  [align=left] {$\displaystyle x_0$};

\end{tikzpicture}

\caption{An edge-indexed graph $(Q,\ind)$ of a non-geometrically finite lattice $\Gamma \leq \Aut(T_{2q+2})$.} 
\label{fig:2q-ray}
\end{figure}
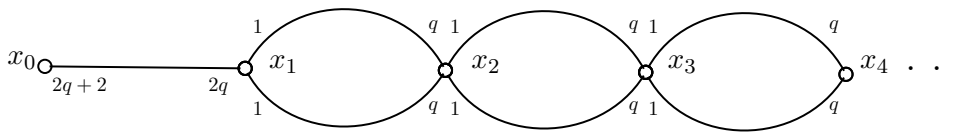

\begin{theorem}\label{thm:escape_of_mass}
Let $\Gamma$ be a tree lattice with associated edge-indexed graph $(Q, \ind)$ as in Fig.~\ref{fig:2q-ray}. Let $x=e\Gamma \in X = G/\Gamma$ be the trivial coset. 
Let $\xi \in \partial T$ be the end corresponding to the sequence $\{\tilde{x}_i\}_{i \in \bbN}$ for some lifts of $x_i$ and $G_\xi^0$ be the corresponding horospherical subgroup. Then for any compact $K\subset X$ 
\[ \lim_{t\to \infty} \nu_{x,t}(K) = 0  . \]
\end{theorem}


\begin{proof}
By \eqref{eq.folners1}, it clearly suffices to prove the statement for the orbital measures $\nu_{x,t}$ associated to the F\o lner sequence $F_t$. Let $\tilde{o}$ be a lift of the left-most vertex $x_0$ to $T=T_{2q+2}$. By Lemma \ref{lem:nu_and_sigma*} and the fact that $\proj$ has compact fibers, it is enough to show that if $\sigma_t$ is the uniform measure on $F_t \tilde{o}\subset VT$, then for any $x_l\in VQ$, we have $\pi_*\sigma_t (x_l) \to 0$. 

The set $F_t\tilde{o}$ can be identified with all non-backtracking paths in $T$ of length $t$ that start at $\tilde{x}_{t}$ and do not contain $\tilde{x}_{t+1}$. A path from $\tilde{x}_t$ to a vertex $y\in F_t\tilde{o}$ with $\pi(y)=x_l$ projects to a path in $Q$ between $x_t$ and $x_l$. Note that the projection of such paths to $Q$ can only contain $x_0$ as endpoint. These will allow us to bound the number of such paths.

Without loss of generality, assume that $t$ is even. For $l$ an even non-negative integer, we claim that the number of vertices in $F_t\tilde{o}$ that project to $x_l$ is bounded above by
\[  {t \choose l/2}  \cdot (2q)^{t-l/2} \cdot 2^{l/2}.  \]
Indeed, any such path from $x_t$ to $x_l$ must take $t-l/2$ steps to the left and $l/2$ steps to the right in Fig.~\ref{fig:2q-ray}. The binomial coefficient counts the number of choices when to take the right step. Since the projection of the paths that we consider can only contain $x_0$ as endpoint, for any choice of such path, each edge taken to the right has at most $2$ lifts to $T$, while each edge taken to the left has at most $2q$ lifts. Therefore, for any even $l\geq 0$
\[   \pi_* \sigma_t(x_l)  \leq \left(\frac{2q}{2q+1}\right)^t g(t) \underset{t \to \infty}{\longrightarrow} 0, \]
where $g(t)$ is a polynomial in $t$ of degree $\leq l/2$. \end{proof}


The rest of this section is devoted to the proof of Theorem \ref{thm.escape.and.equidist}. Its proof consists of four parts. In the first part, we construct an uncountable family of lattices $\Gamma_\alpha$ in $\Aut(T_6)$. In the second part, thanks to an auxiliary Markov chain that we introduce, we obtain subgaussian concentration estimates on the Markov chain associated with the lattice $\Gamma_\alpha$ (see \S \ref{subsec:mc-def}). In the third part, we show that the space $\Aut(T_6)/\Gamma_\alpha$ contains points $x$ which exhibit escape of mass along some subsequence of horospherical orbital averages and along some other subsequences equidistribute to the Haar measure, as we show in the fourth part. 

\begin{proof}[Proof of Theorem \ref{thm.escape.and.equidist}]

\textit{First part: Construction of $\Gamma_\alpha$.} For each $\alpha \in (1,2)$, we will construct  an edge-indexed graph $(Q_\alpha,\ind)$ of finite volume, which will yield a lattice $\Gamma_\alpha \leq \Aut(T_6)$.
First, the underlying graph is a ray, with vertices $\{x_i\}_{i=0}^\infty$ and edges $e_i = (x_i,x_{i+1})$.

Let $\alpha \in (1,2)$ and for $i \geq 1$, let $n_i = \lfloor \alpha^i \rfloor$. We divide the vertices $(x_i)_{i \geq 1}$ of the ray into two types: $x_j$ is \textbf{black} if $j=i+n_1+ \cdots +n_i$ for some $i \geq 1$ and \textbf{white} otherwise. In other words, there are blocks of $n_i$ consecutive white vertices that are separated by single appearances of black vertices. A white vertex $x_j$ is said to belong to $i-$th \textbf{block} if
\[  i+n_1 + \cdots n_i < j < i +1 + n_1 + \cdots +n_{i+1}. \]
We say that an edge $e$ belongs to the $i$-th block if both $\partial_0 e$ and $\partial_1 e$ do.

We define the index map on $EQ_\alpha$:  for $i>0$, let $\ind(e_i)=2$ and $\ind(\bar{e}_{i-1})=4$ if $x_i$ is black and $\ind(e_i)=\ind(\bar{e}_{i-1})=3$ if $x_i$ is white. Set $\ind(e_0)=6$. See Fig.~\ref{fig:gamma_alpha} for illustration.

\tikzset{every picture/.style={line width=0.75pt}} 
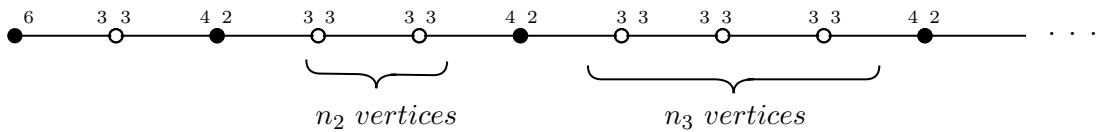
\begin{figure}[H]
    \centering

\begin{tikzpicture}[x=0.75pt,y=0.75pt,yscale=-1,xscale=1]

\draw    (29.5,130) -- (78.15,130) ;
\draw [shift={(80.5,130)}, rotate = 0] [color={rgb, 255:red, 0; green, 0; blue, 0 }  ][line width=0.75]      (0, 0) circle [x radius= 3.35, y radius= 3.35]   ;
\draw [shift={(29.5,130)}, rotate = 0] [color={rgb, 255:red, 0; green, 0; blue, 0 }  ][fill={rgb, 255:red, 0; green, 0; blue, 0 }  ][line width=0.75]      (0, 0) circle [x radius= 3.35, y radius= 3.35]   ;
\draw    (82.85,130) -- (129.15,130) ;
\draw [shift={(131.5,130)}, rotate = 0] [color={rgb, 255:red, 0; green, 0; blue, 0 }  ][line width=0.75]      (0, 0) circle [x radius= 3.35, y radius= 3.35]   ;
\draw [shift={(80.5,130)}, rotate = 0] [color={rgb, 255:red, 0; green, 0; blue, 0 }  ][line width=0.75]      (0, 0) circle [x radius= 3.35, y radius= 3.35]   ;
\draw    (131.5,130) -- (180.15,130) ;
\draw [shift={(182.5,130)}, rotate = 0] [color={rgb, 255:red, 0; green, 0; blue, 0 }  ][line width=0.75]      (0, 0) circle [x radius= 3.35, y radius= 3.35]   ;
\draw [shift={(131.5,130)}, rotate = 0] [color={rgb, 255:red, 0; green, 0; blue, 0 }  ][fill={rgb, 255:red, 0; green, 0; blue, 0 }  ][line width=0.75]      (0, 0) circle [x radius= 3.35, y radius= 3.35]   ;
\draw    (184.85,130) -- (231.15,130) ;
\draw [shift={(233.5,130)}, rotate = 0] [color={rgb, 255:red, 0; green, 0; blue, 0 }  ][line width=0.75]      (0, 0) circle [x radius= 3.35, y radius= 3.35]   ;
\draw [shift={(182.5,130)}, rotate = 0] [color={rgb, 255:red, 0; green, 0; blue, 0 }  ][line width=0.75]      (0, 0) circle [x radius= 3.35, y radius= 3.35]   ;
\draw    (235.85,130) -- (284.5,130) ;
\draw [shift={(284.5,130)}, rotate = 0] [color={rgb, 255:red, 0; green, 0; blue, 0 }  ][fill={rgb, 255:red, 0; green, 0; blue, 0 }  ][line width=0.75]      (0, 0) circle [x radius= 3.35, y radius= 3.35]   ;
\draw [shift={(233.5,130)}, rotate = 0] [color={rgb, 255:red, 0; green, 0; blue, 0 }  ][line width=0.75]      (0, 0) circle [x radius= 3.35, y radius= 3.35]   ;
\draw    (286.85,130) -- (333.15,130) ;
\draw [shift={(335.5,130)}, rotate = 0] [color={rgb, 255:red, 0; green, 0; blue, 0 }  ][line width=0.75]      (0, 0) circle [x radius= 3.35, y radius= 3.35]   ;
\draw [shift={(284.5,130)}, rotate = 0] [color={rgb, 255:red, 0; green, 0; blue, 0 }  ][line width=0.75]      (0, 0) circle [x radius= 3.35, y radius= 3.35]   ;
\draw    (337.85,130) -- (384.15,130) ;
\draw [shift={(386.5,130)}, rotate = 0] [color={rgb, 255:red, 0; green, 0; blue, 0 }  ][line width=0.75]      (0, 0) circle [x radius= 3.35, y radius= 3.35]   ;
\draw [shift={(335.5,130)}, rotate = 0] [color={rgb, 255:red, 0; green, 0; blue, 0 }  ][line width=0.75]      (0, 0) circle [x radius= 3.35, y radius= 3.35]   ;
\draw    (388.85,130) -- (435.15,130) ;
\draw [shift={(437.5,130)}, rotate = 0] [color={rgb, 255:red, 0; green, 0; blue, 0 }  ][line width=0.75]      (0, 0) circle [x radius= 3.35, y radius= 3.35]   ;
\draw [shift={(386.5,130)}, rotate = 0] [color={rgb, 255:red, 0; green, 0; blue, 0 }  ][line width=0.75]      (0, 0) circle [x radius= 3.35, y radius= 3.35]   ;
\draw    (439.85,130) -- (486.15,130) ;
\draw [shift={(488.5,130)}, rotate = 0] [color={rgb, 255:red, 0; green, 0; blue, 0 }  ][line width=0.75]      (0, 0) circle [x radius= 3.35, y radius= 3.35]   ;
\draw [shift={(437.5,130)}, rotate = 0] [color={rgb, 255:red, 0; green, 0; blue, 0 }  ][line width=0.75]      (0, 0) circle [x radius= 3.35, y radius= 3.35]   ;
\draw    (488.5,130) -- (539.5,130) ;

\draw [shift={(488.5,130)}, rotate = 0] [color={rgb, 255:red, 0; green, 0; blue, 0 }  ][fill={rgb, 255:red, 0; green, 0; blue, 0 }  ][line width=0.75]      (0, 0) circle [x radius= 3.35, y radius= 3.35]   ;
\draw   (176.5,142) .. controls (176.43,146.67) and (178.73,149.03) .. (183.4,149.1) -- (201.9,149.36) .. controls (208.57,149.45) and (211.87,151.83) .. (211.8,156.5) .. controls (211.87,151.83) and (215.23,149.55) .. (221.9,149.64)(218.9,149.6) -- (240.4,149.9) .. controls (245.07,149.97) and (247.43,147.67) .. (247.5,143) ;
\draw   (318.5,145) .. controls (318.5,149.67) and (320.83,152) .. (325.5,152) -- (382,152) .. controls (388.67,152) and (392,154.33) .. (392,159) .. controls (392,154.33) and (395.33,152) .. (402,152)(399,152) -- (458.5,152) .. controls (463.17,152) and (465.5,149.67) .. (465.5,145) ;

\draw (36.5,120.39) node  [font=\tiny] [align=left] {$\displaystyle 6$};
\draw (136.5,120.39) node  [font=\tiny] [align=left] {$\displaystyle 2$};
\draw (125.5,120.39) node  [font=\tiny] [align=left] {$\displaystyle 4$};
\draw (73.5,120.39) node  [font=\tiny] [align=left] {$\displaystyle 3$};
\draw (85.5,120.39) node  [font=\tiny] [align=left] {$\displaystyle 3$};
\draw (177.5,120.39) node  [font=\tiny] [align=left] {$\displaystyle 3$};
\draw (188.5,120.39) node  [font=\tiny] [align=left] {$\displaystyle 3$};
\draw (228.5,120.39) node  [font=\tiny] [align=left] {$\displaystyle 3$};
\draw (239.5,120.39) node  [font=\tiny] [align=left] {$\displaystyle 3$};
\draw (290.5,120.39) node  [font=\tiny] [align=left] {$\displaystyle 2$};
\draw (279.5,120.39) node  [font=\tiny] [align=left] {$\displaystyle 4$};
\draw (432.5,120.39) node  [font=\tiny] [align=left] {$\displaystyle 3$};
\draw (443.5,120.39) node  [font=\tiny] [align=left] {$\displaystyle 3$};
\draw (335.5,120.39) node  [font=\tiny] [align=left] {$\displaystyle 3$};
\draw (346.5,120.39) node  [font=\tiny] [align=left] {$\displaystyle 3$};
\draw (377.5,120.39) node  [font=\tiny] [align=left] {$\displaystyle 3$};
\draw (388.5,120.39) node  [font=\tiny] [align=left] {$\displaystyle 3$};
\draw (493.5,120.39) node  [font=\tiny] [align=left] {$\displaystyle 2$};
\draw (482.5,120.39) node  [font=\tiny] [align=left] {$\displaystyle 4$};
\draw (217,171) node   [align=left] {$\displaystyle n_2 \ vertices$};
\draw (393,171) node   [align=left] {$\displaystyle n_3 \ vertices$};
\draw (563,129) node   [align=left] {. . .};

\end{tikzpicture}
\caption{Edge-indexed graph $(Q_\alpha, \ind)$. Second and third blocks are marked.}
\label{fig:gamma_alpha}
\end{figure}

One easily checks that $(Q_\alpha,\ind)$ has bounded denominators (\cite[page 23]{bass-lubotzky}) and, hence, has a faithful finite grouping which yields a discrete subgroup $\Gamma_\alpha$ in $\Aut(T_{6})$. Moreover, $\vol(Q_\alpha) \leq 1+ 4 \sum_{i \geq 1} (n_i +1)\frac{1}{2^{i-1}}  < \infty$ (see \eqref{eq:volume-formula}), thus $\Gamma_\alpha$ is a lattice.\\

\textit{Second part: An auxiliary chain and subgaussian  estimates.} Consider the edge-indexed graph $Q_\alpha$ and the Markov chain $M_n$ on $EQ$ as in $\S \ref{subsec:mc-def}$.  
For an edge $e_j$ that belongs to some block $i$, the transition probabilities of $M_n$ are given by
\begin{equation}\label{eq.b.transitions}
P(e_j,e_{j+1})=P(\bar{e}_j,\bar{e}_{j-1})=\frac{3}{5}, \qquad
P(e_j,\bar{e}_j)=P(\bar{e}_j,e_j)=\frac{2}{5}.
\end{equation}

In view of the reductions in Section \ref{sec.rec}, we are interested in understanding the distribution of $\partial_0 {}_\ast \delta_{e_j} P^m$. For this, consider an auxiliary Markov kernel $\bar{P}$ on a state space consisting of two elements $\{e,\bar{e}\}$ with transition probabilities
\begin{equation}\label{eq.micky.transitions}
\bar{P}(e,e)=\bar{P}(\bar{e},\bar{e})=\frac{3}{5}, \qquad \bar{P}(e,\bar{e})=\bar{P}(\bar{e},e)=\frac{2}{5}. 
\end{equation}
This auxiliary Markov Chain records the behavior of $M_n$ along the edges within some block. It remembers the probabilities of an edge to turn around or to continue further in the same direction (see e.g.\ Fig.~\ref{fig:nagaochain}).

Denote by $V_n$ the Markov chain associated to the kernel $\bar{P}$. Given a word $\mathrm{u} \in \{e,\bar{e}\}^n$ of length $n \geq 2$, for $ s \in \{e,\bar{e}\}^2$ denote by $N_s(\mathrm{u})$ the number of occurrences of the word $s$ as a subword of $\mathrm{u}$. For $n \geq 2$, define the function $f_n$ on $\{ e,\bar{e} \}^n$ by $f_n(\mathrm{u})=N_{ee}(\mathrm{u})+N_{e\bar{e}}(\mathrm{u})-N_{\bar{e}e}(\mathrm{u})-N_{\bar{e}\bar{e}}(\mathrm{u})$. Denote by $Y_n$ the integer valued random variable $f_n(V_0,V_1,\ldots,V_{n-1})$. 

Now, it is readily observed that for every $i \in \mathbb{N}$ large enough so that $n_i\geq 16$ and every integer $j,m \geq 2$ with 
\begin{equation}\label{eq.interval.constraints}
\begin{aligned}
i+\sum_{k=1}^{i-1}n_k + \frac{n_i}{4}  \leq j \leq i+\sum_{k=1}^{i}n_k -  \frac{n_i}{4} \qquad \text{and} \qquad  m  \leq \frac{n_i}{8},
\end{aligned}
\end{equation}  we have
\begin{equation}\label{eq.express.with.micky}
\partial_0 {}_\ast \delta_{e_j}P^m = \partial_0 {}_\ast e_{j+Y_m}  \quad \text{in distribution.}
\end{equation}

Indeed, the inequalities \eqref{eq.interval.constraints} make sure that the starting edge $e_j$ in the $i$-th block is $n_i/4$ away from the boundary of the $i$-th block, so application of $P^m$ keeps the support of the distribution within $i$-th block and, therefore, transition probabilities at each step are given by \eqref{eq.b.transitions}. 
The relation with the auxiliary chain \eqref{eq.micky.transitions} is straightforward.

We now wish to use subgaussian concentration inequalities  for the Markov chain $V_m$ (e.g.~as discussed in \cite{dedecker-gouezel}). To this end, we note that being an aperiodic irreducible Markov chain with finite state space $\{e,\overline{e}\}$, $V_m$ is geometrically ergodic and the state space is a small set in the sense of \cite[Definition 0.1]{dedecker-gouezel}. Moreover, one checks by a simple calculation that $|\mathbb{E}_{s}[Y_m]|$ is bounded by a constant $C$ independently of $s \in \{e,\bar{e}\}$ and $m \in \mathbb{N}$. Finally, each function $f_m$ is clearly separately bounded in the sense of \cite[(0.1)]{dedecker-gouezel}. Now, it follows by \cite[Theorem 0.2, (0.7)]{dedecker-gouezel}) that there exists a constant $c'_0>0$ such that for all $m \in \mathbb{N}$, $s \in \{e,\overline{e}\}$ and $r \in \mathbb{N}$, we have
\begin{equation*}
\mathbb{P}_{s}(|Y_m| \leq C+r) \geq 1-2 e^{-c'_0\frac{r^2}{m}}.
\end{equation*}
Now using the relation \eqref{eq.express.with.micky} and slightly decreasing the constant $c_0'$ to $c_0$ (depending only on $C>0$), we deduce that for every $i \in \mathbb{N}$ large enough, $j,m \in \mathbb{N}$ as in \eqref{eq.interval.constraints} and $r \leq m$,  we have
\begin{equation}\label{eq.subgaussian.estimate}
\mathbb{P}_{e_{j}}(M_m \in \{e_{j-r},\ldots,e_{j+r},\bar{e}_{j-r},\ldots,\bar{e}_{j+r}\}) >1- 2 e^{-c_0\frac{r^2}{m}}.
\end{equation}

It follows immediately that if an initial distribution $\mu$ is supported on a set $S$ of edges $\{e_j\}$ which satisfy \eqref{eq.interval.constraints} for some $i$ large enough, then for $m$ as in \eqref{eq.interval.constraints} and $r\leq m$ we have
\begin{equation}\label{eq.subgaussian.support}
\mathbb{P}_{\mu}(M_m \text{ is in }r\text{-neighborhood of }S) >1- 2 e^{-c_0\frac{r^2}{m}}.    
\end{equation}


\textit{Third part: Showing the escape of mass.}
Now we construct points $x = g\Gamma_\alpha \in G/\Gamma_\alpha$ who, under horospherical group action, exhibit the dynamical behavior as described in the statement of Theorem \ref{thm.escape.and.equidist}.
For an edge $e\in EQ_\alpha$, denote by $|e| = d(e,x_0)$ the graph distance in $Q_\alpha$. In particular, $|e_j|=j$. 

Choose a sequence of edges $e(t) \subset EQ_\alpha$, such that for some $\beta \in (\frac{2}{3},1)$ we have
\begin{enumerate}
\item $e(0)=e_0$.
    \item $\partial_0 e_{t+1}= \partial_1 e(t)$.
    \item $\liminf_{t\to \infty} \frac{t^\beta}{|e(t)|} =  0.    $
    \item $|e(t)|=0$ for infinitely many $t\in \bbN$.
\end{enumerate}

The path $e(t)$ as above comes back infinitely often to $x_0$, but makes some longer and longer visits towards the cusp.
Now, choose a lift of $e(t)$ in $T_6$, starting at some basepoint $\tilde{o} \in VT$, which is the lift of $x_0$. Let $\tilde{o}=y_0, y_1, y_2,...$ be consecutive vertices converging to $\eta \in \partial T$. Let $g\in \Aut(T_6)$ be an automorphism such that $g^{-1}$ maps the edge $(y_i,y_{i+1})$ to the lift of the edge $e(i)$, and $x=g\Gamma_\alpha$.

Let $i_k$ and $t_{i_k}$ be increasing sequences of $\bbN$, such that $e(t_{i_k})$ is an edge, pointing toward the cusp, that is exactly in the middle of $i_k$-th block, namely
\[d_{i_k}:= |e(t_{i_k})| = i + n_1 + ... +n_{i_k-1} + \lfloor \frac{n_{i_k}}{2} \rfloor ,\]
and $t_{i_k}^\beta < c_1 |e(t_{i_k})|$ for some $c_1>0$. Such infinite subsequences exist by property (3) in the choice of $e(t)$.
In the notation above, $e_{d_{i_k}}=|e(t_{i_k})|$.


We shall now show that the sequence of measures given by the orbital averages $\nu_{x,t_{i_k}}$ converges weakly to $0$ as $k\to \infty$. 

By Lemma \ref{lem:nu_and_sigma*} and \eqref{eq5}, it suffices to show that the sequence $\delta_{e(t_{i_k})}P^{t_{i_k}}$ of distributions of the Markov chain $M_n$ converges weakly to zero. 
To do this, we would like to apply \eqref{eq.subgaussian.estimate} to show that after $t_{i_k}$-iteration most of the mass of the Markov chain stays in $i_k$-th block, which moves to the cusp in $Q_\alpha$ as $k \to \infty$. However, constraints $(\ref{eq.interval.constraints})$ are not satisfied since $t_{i_k}\geq d_{i_k} > n_{i_k}/8$. We remind that $n_{i_k} = \lfloor \alpha^{i_k} \rfloor$, $d_{i_k} \sim c_2\alpha^{i_k}$ and, thus  $t_{i_k} \leq  c_3 \left( \alpha^{i_k}\right)^{1/\beta} $ for some positive constants $c_2,c_3$.

To overcome this problem, we apply \eqref{eq.subgaussian.support} several times for a small number of allowed iterations, each time dismissing an exponentially small proportion of trajectories that move more than a distance $r_k$ to be chosen below.

Let $m_k =\lfloor \frac{n_{i_k}}{8} \rfloor \sim \frac{\alpha^{i_k}}{8}$. The number of times we wish to apply \eqref{eq.subgaussian.support} is bounded above by
\[ N_k =\lceil 8c_3 \alpha^{i_k(1/\beta -1)} \rceil \geq  \frac{t_{i_k} }{m_k}   \]
Set $r_k  = \lceil \frac{n_{i_k}/4}{N_k}  \rceil \sim (32c_3)^{-1} \alpha^{i_k(2-1/\beta)}$. The Markov property and choices of $m_k$ and $r_k$ allow us to repeatedly apply \eqref{eq.subgaussian.support} $N_k$ times (with $m=m_k$ and $r=r_k$), each time conditioning on trajectories that do not move more than $r_k$ in each $m_k$-iterate. We get that proportion of trajectories starting at $e(t_{i_k})$ that move at most $N_k \cdot r_k \leq n_{i_k}/4$ (in particular, do not leave $i_k$-th block) is at least
\[   \left( 1-2e^{-c_4\alpha^{i_k(3-2/\beta)}} \right)^{8c_2 \alpha^{i_k(1/\beta -1)}},  \]
for some constant $c_4>0$. The above tends to $1$ as $k\to \infty$, implying the escape of mass for $\nu_{x,t}$'s when the underlying F\o lner sequence is $F_t$. By \eqref{eq.folners1}, this also implies the escape of mass for $\nu_{x,t}$'s associated to any good F\o lner sequence.\\

\textit{Fourth part: Equidistribution.} Recall that by our choice of $x \in X$, there exists an increasing sequence $t_k \in \mathbb{N}$ such that $|e(t_k)|=0$ for every $k \in \mathbb{N}$. The equidistribution statement follows from the following technical but more general result. This completes the proof of Theorem \ref{thm.escape.and.equidist}. \end{proof}

\begin{proposition}
Let $G$ be a non-compact, closed, topologically simple subgroup of $\Aut(T)$ that acts transitively on $\partial T$. Let $\Gamma$ be a lattice in $G$, $\eta \in \partial T$ and $O_t$ a good F\o lner sequence in $G^0_\eta$. Let $g \in G$ and denote by $(\hat{e}(t))_{t \in \mathbb{N}}$ a sequence of consecutive edges in $T$ on a geodesic segment towards $g^{-1} \eta$. Assume that there exist a finite subset $F$ of $EQ$ and an increasing subsequence $t_k$ such that $\pi (\hat{e}(t_k))=:e(t_k) \in F$ for every $k \in \mathbb{N}$. Then for $x=g \Gamma \in X$, the orbital measures $\nu_{x,t_k}$ converge towards the Haar measure $m_X$ on $X$.
\end{proposition}


\begin{proof}
As before, fix a distinguished vertex $\tilde{o}$ in $T$ with respect to which the $\proj: G/\Gamma \to VQ$ map given by $\proj(h\Gamma)=\pi(h^{-1} \tilde{o})$ is defined. Let $g \in G$ be as in the statement. Since the Markov chain $M_n$ associated with the lattice $\Gamma$ is positive recurrent (Proposition \ref{prop.pos.rec}), it follows by \eqref{eq.folners1} and the correspondance established in Lemmas \ref{lem:nu_and_sigma*}, \ref{lemma.side.sphere} and \ref{lemma.mc.spheres} (see also Remark \ref{rk.general}) that $\proj_\ast \nu_{x,t_k}$ converges to a probability measure $\bar{m}$ on $EQ$. Since the map $\proj$ is proper, this implies that the sequence $\nu_{x,t_k}$ is tight so that any subsequence of $(\nu_{x,t_k})_{k \in \mathbb{N}}$ has a limit point and any limit point is a probability measure on $X$. Let $m$ be such a limit point along a subsequence that we also denote by $t_k$. Since $\nu_{x,t_k}$'s are orbital measures associated to a F\o lner sequence in $G^0_\eta$, the limit probability measure $m$ is $G^0_\eta$-invariant.

Now, fix a hyperbolic element $a \in G$ with attracting point $\eta \in \partial T$ and such that the translation axis of $a$ contains $\tilde{o}$. Let $n_k=\lfloor \frac{t_k}{\tau(a)} \rfloor$ where $\tau(a) \in \mathbb{N}$ denotes the translation length of $a$, so that $|\tau(a^{n_k})-t_k|$ is bounded. For every $k \in \mathbb{N}$, we have $\proj(a^{-n_k}g\Gamma)=\pi(g^{-1}a^{n_k} \tilde{o})=\pi((g^{-1}a^{n_k} g)g^{-1}\tilde{o})$. As $n \in \mathbb{N}$ varies, $(g^{-1}a^n g)g^{-1}\tilde{o}$ describes  vertices on the geodesic ray between $g^{-1} \tilde{o}$ and $g^{-1} \eta$. Therefore it follows by the hypothesis $e(t_k) \in F$ that for some larger finite set $F' \subset EQ$, we have $\proj(a^{-n_k}g\Gamma) \in F'$ for every $k \in \mathbb{N}$. Since the map $\proj$ has compact fibres, this entails that there exists a compact set $K \subset G/\Gamma$ such that 
\begin{equation}\label{eq.in.K}
a^{-n_k}g \Gamma \in K,    
\end{equation}
for every $k \in \mathbb{N}$. 
Furthermore, since such a group $G$ as in the statement enjoys the Howe--Moore property \cite{burger-mozes.product} (see also \cite{lubotzky-mozes}), the action of $a$ on $(X,m_X)$ is mixing so that we are in a position to apply \cite[Lemma 6.3]{CFS} as in the proof of Proposition \ref{prop.geofin.measure}. Now repeating the same argument as in the end of the proof of Proposition \ref{prop.geofin.measure} (i.e.\ \eqref{eq.howe.moore} and thereafter), one deduces that $m=m_X$ and this proves the proposition.
\end{proof}

\section{Limiting distributions of spheres in quotient graphs and lattice point counting}\label{sec.spheres}

This section is devoted to the proof of Theorems \ref{thm.dist.spheres} and \ref{thm.counting}   . Recall from  \S \ref{subsec:mc-def}, the irreducible Markov chain $M_n$ associated with a tree lattice $\Gamma$. We proved in Lemma \ref{lem:recur-geomfin} that it is positive recurrent when $\Gamma$ is a geometrically finite lattice. However, in Theorem \ref{thm.dist.spheres} general lattices are considered. Here we will prove that, more generally, $M_n$ is positive recurrent for all lattices. In order to do this we introduce another Markov chain that will serve as a tool to analyse further the chain $M_n$.

\subsection{Auxiliary chain}
Consider the Markov chain $\hat{M}_n$ on the state space $VQ$, given by the transition kernel $\hat{P}(v,w)=\frac{\ind((v,w))}{\deg(v)}$ if $(v,w) \in EQ$ and $0$ otherwise. Since the graph $Q$ is connected, $\hat{M}_n$ is irreducible.


Recall that $\tilde{o}\in VT$ is a fixed basepoint. Let $o=\pi(\tilde{o})\in VQ$. Consider the positive function $\mu:VQ \to \bbR_+$, given by $\mu(w)=\deg(w)N_o(w)^{-1} $, where $N_o(.)$ is as defined in \S \ref{subsec.edge.indexed}.

\begin{lemma}\label{lem:aux-recu}
The positive function $\mu$ defines a finite stationary measure on $VQ$ for $\hat{M}_n$. In particular, the Markov chain $\hat{M}_n$ is positively recurrent.
\end{lemma}

\begin{proof}
It suffices to check that $\mu$ has finite $l_1$-norm and is reversible, i.e.\ satisfies $\mu(w_1)\hat{P}(w_1,w_2)=\mu(w_2)\hat{P}(w_2,w_1)$ for all $w_1,w_2 \in VQ$. It is enough to consider pairs of neighbors $w_1,w_2 \in VQ$. Indeed, for such we have
$$\frac{\hat{P}(w_1,w_2)}{\hat{P}(w_2,w_1)}=\frac{\ind(w_1,w_2)}{\ind(w_2,w_1)}\frac{\deg(w_2)}{\deg(w_1)}=\frac{N_o(w_1)}{N_o(w_2)}\frac{\deg(w_2)}{\deg(w_1)}=\frac{\mu(w_2)}{\mu(w_1)}.$$ This shows that $\mu$ is a reversible measure on $VQ$. The fact that $\mu$ has finite $l_1$-norm is a direct consequence of the volume formula \eqref{eq:volume-formula}: 
$$
\sum_{w \in VQ} \mu(w)= \sum_w \deg(w)N_o(w)^{-1} \leq \max\{d_1,d_2\} \sum_w N_o(w)^{-1} < \infty.
$$
\end{proof}

\begin{remark}\label{rem:identify-mu}
Recall from \S \ref{subsec.edge.indexed} that $G$ is transitive on the set of vertices of $VT$ at even distance from $\tilde{o}$. The image of $\proj: G/\Gamma \to \Gamma \setminus T$ is the set of vertices at even distance from $o$. Moreover, from \eqref{eq:measure-of-po-fiber} it is clear that  $\proj_*m_X$ is proportional to the restriction of $\mu$ to the image of $\proj$. Hence, the measure $\mu$ can be thought of as the projection of Haar measure on $G/\Gamma$. 
\end{remark}

\subsection{Positive recurrence of the Markov chain $M_n$}\label{subsec:posrec}

First, we wish to relate the Markov chains $\hat{M}_n$ and $M_n$. Denote by $R_n$ the $n^{\text{th}}$-step of the  nearest-neighbor simple random walk on the vertices of the tree $T$ and by $\delta_{\tilde{o}} (R_n)$ its distribution when the initial vertex is $\tilde{o} \in VT$, i.e.\ a.s.\ $R_0=\tilde{o}$. Note that since $T$ is biregular, the restriction of $\delta_{\tilde{o}}(R_n)$ to the spheres $S(\tilde{o},m)$, for $m \leq n$, is a multiple of the uniform measure on $S(\tilde{o},m)$ which is denoted by $\rho_m$ as before. Let $D_{\tilde{o}}$ be the distribution on $EQ$ given by



\begin{equation}\label{eq.def.D}
    D_{\tilde{o}}:= \frac{1}{\deg(\tilde{o})} \sum_{(\tilde{o},\tilde{w})\in ET} \delta_{(o,\pi(\tilde{w}))}. 
\end{equation}  




\begin{lemma}\label{lemma.aux.main}
For any $n\geq 1$ we have
\begin{equation*}
\begin{aligned}
\delta_o \hat{P}^n&=\pi_\ast  \delta_{\tilde{o}}(R_n)\\
&=\mathbb{P}_{\tilde{o}}(R_n=\tilde{o})\delta_o+ \sum_{k=1}^n \mathbb{P}_{\tilde{o}}(d(R_n,\tilde{o})=k)
\partial_{1*}(D_{\tilde{o}} P^{k-1}).
\end{aligned} 
\end{equation*}
\end{lemma}

\begin{proof}
Let $R$ be the transition kernel for simple random walk on the tree.  For the first equality, one simply notes that $\hat{P}(\pi(\tilde{x}),\pi(\tilde{y}))=R(\tilde{x},\pi^{-1}(\tilde{y}))$.

The second equality follows from the fact that the distribution of $n^{th}$-step of nearest-neighborhood simple random walk on the tree is given by
\begin{equation}
    \delta_{\tilde{o}}(R_n)=  \mathbb{P}_{\tilde{o}}(R_n=\tilde{o}) \delta_{\tilde{o}} + \sum_{k=1}^n \mathbb{P}_{\tilde{o}}(d(R_n,\tilde{o})=k) \rho_k.
\end{equation}  
The statement follows after applying $\pi_*$ and Lemma \ref{lemma.mc.spheres}. 
\end{proof}

In other words, the distribution of the chain $\hat{M}_n$ starting from $v \in VQ$ is given by weighted average of distributions given by $M_k$ with $k\leq n$. We will use this relation to deduce the positive recurrence of $M_n$ from the positive recurrence of $\hat{M}_n$.

\begin{proposition}\label{prop.pos.rec}
The Markov chain $M_n$ is positive recurrent.  
\end{proposition}
\begin{proof}

By Kingman's subadditive ergodic theorem, there exists $r\in \bbR$ such that $\frac{1}{k} d(\tilde{o},R_k ) \longrightarrow r$, $\mathbb{P}_{\tilde{o}}$-almost surely, hence also in measure,  as $k\to \infty$ (the value $r$ is called the drift of random walk $R_k$). 
Since $ \max \{d_1,d_2\} \geq 3$, it is easily seen that $r>0$.  Let $\varepsilon >0$. Then for all $k \in \mathbb{N}$ large enough, we have
\begin{equation}\label{eq1}
\mathbb{P}_{\tilde{o}}\left( | d(\tilde{o},R_k)-kr|>k\varepsilon\right) < \varepsilon
\end{equation}

By positive recurrence of the auxiliary chain $\hat{M}_n$ (Lemma \ref{lem:aux-recu}), there exists a finite subset $K_1$ of $VQ$ such that for every $n$ large enough, $\mathbb{P}_o (\hat{M}_n \in K_1)>1-\epsilon$. 

In view of Lemma \ref{lemma.aux.main} and $(\ref{eq1})$, we deduce that there exists a sequence $n_k \in \mathbb{N}$ with $|n_k-kr|\leq k\varepsilon$ such that for every $k$ large enough  
\begin{equation}\label{eq2}
\mathbb{P}(M_{n_k} \in \partial_1^{-1}K_1)>1-2\epsilon,
\end{equation}
which implies that the irreducible chain $M_n$ is positive recurrent.
\end{proof}

An alternative and more conceptual proof of Proposition \ref{prop.pos.rec} was kindly suggested to us by an anonymous referee. We discuss it in the following remark. As our proof above, it relies on the fact that the Markov chain $M_n$ can be seen as a quotient of the simple random walk on $ET$, the set of edges of the tree $T$.

\begin{remark}[Alternative proof of Proposition \ref{prop.pos.rec}]
Let $\tilde{P}$ be the Markov operator associated to the simple random walk on $ET$. Considering two successive edges $x$ and $y$ in $ET$, we have $ET=Gx \cup Gy \simeq G/G_x \cup G/G_y$, where $G_x$ and $G_y$ denote the respective stabilizers and $G=\Aut(T)$. Using this and the fact that $G$ is unimodular \cite[Proposition 6]{Amann}, one sees that $ET$ carries a $\tilde{P}$-stationary and $G$-invariant measure $\tilde{\nu}$. The restriction of $\tilde{\nu}$ to $Gx$ (respectively $Gy$) corresponds to the $G$-invariant measure on $G/G_x$ (respectively $G/G_y$). On the other hand, the Markov operator $P$ of the Markov chain $M_n$ on $EQ \simeq \Gamma \setminus ET$ can be seen as the restriction of $\tilde{P}$ to $\Gamma$-invariant functions on $ET$ and the associated quotient measure $\nu$ of $\tilde{\nu}$ gives a $P$-stationary measure on $EQ$. But since $\Gamma<G$ is a lattice and $\nu$ is given by the quotient measure on $\Gamma \setminus G/G_x \cup \Gamma \setminus G/G_y$, we have that $\nu$ is finite, as required.
\end{remark}

\subsection{Proof of Theorem \ref{thm.dist.spheres}}\label{subsec.proof.of.C}
Here we prove parts 1.\ and 2.\ of Theorem \ref{thm.dist.spheres}. Its third part about exponential equidistribution will be proven in \S \ref{subsec.effective}.

\bigskip

Given $\tilde{v} \in VT$, let $D_{\tilde{v}}$ the distribution on $VQ$ defined as in \eqref{eq.def.D}. By Lemma \ref{eq:sphere-sum_of_pacman}, $\pi_* \rho_n =\partial_{1*}(D_{\tilde{v}}P^{n-1})$. Hence, by Proposition \ref{prop.pos.rec}, there is no escape of mass for the sequence $\pi_\ast \rho_n$. This proves $(1)$ of Theorem \ref{thm.dist.spheres}.

If the irreducible and positive recurrent Markov chain $M_n$ has period $p \in \mathbb{N}$, then the sequence of distributions $D_{\tilde{v}}P^n$ have finitely many limit points $\{\mu_j\}_{j=0}^{p'-1}$, corresponding to all possible convex combinations with coefficients $1/\deg(\tilde{v})$ of the unique stationary probability measures of $M_n$ on each one of its cyclic classes (corresponding to the classes of Dirac measures constituting $D_{\tilde{v}}$). This implies the convergence along subsequences $pn+j$ and hence $(2)$ of Theorem \ref{thm.dist.spheres}.
\qed



\subsection{Exponential equidistribution of spheres in quotients by geometrically finite lattices}\label{subsec.effective}

Previously, we established positive recurrence of $M_n$, which is sufficient to prove the existence of limiting distributions of spheres in quotients of trees by action of tree lattices. However, in some cases, our Markov chain possesses a stronger property, namely that of geometric ergodicity. In these situations, the speed of convergence to the limiting distribution can be shown to be exponential and the exponential rate can even be made effective.

We begin by stating a version of Geometric Ergodic Theorem for Markov chains. Out of the equivalent definitions of geometric ergodicity, we conveniently choose one that uses the (Foster--Lyapunov) drift criteria. We then prove geometric ergodicity of the Markov chain $M_n$ associated to geometrically finite tree lattices  and discuss the application for exponential equidistribution of spheres.  We refer the reader to \cite{meyn-tweedie,DMPS} for more on geometric ergodicity.

Let $M_n$ be an irreducible, aperiodic and  positive recurrent Markov chain on a countable state space $S$ 
with the stationary probability measure $\mu$. Denote by $P$ the corresponding Markov operator. We call $M_n$ \textit{geometrically ergodic} if
there exists $r>1$ such that for all $x\in S$, we have
\begin{equation}\label{eq:geom-erg}
 \sum_{n \geq 0} r^n \|\delta_x P^n - \mu \| < \infty.  
\end{equation}
where $\|\cdot \|$ denotes the total variation norm. In particular, for a geometrically ergodic chain $M_n$, we have $\|\delta_x P^n - \mu \| = o(r^{-n})$ for every $x \in S$.
 
\begin{theorem} (Geometric Ergodic Theorem)
Let $M_n$ be an irreducible aperiodic Markov chain on a  countable state space $S$. Assume that there exist a finite set $K \subset S$, $b\in \bbR, \beta < 1$ and a function $V\geq 1$, which is finite at some $x_0\in S$ satisfying the \textit{drift criteria}:
    \begin{equation}\label{eq:driftcrit}
    PV(x) \leq \beta V(x) + b \mathbbm{1}_K(x), \quad \text{ for any } x\in X.
    \end{equation}
Then $M_n$ is geometrically ergodic.
 \end{theorem}

Let us remark that the rate $r$ can be made explicit in terms of $\beta, K$; see \cite{baxendale} for the treatment of the constant $r$. Finally, aperiodicity hypothesis is only required to have a simple expression as in \eqref{eq:geom-erg}; if the Markov chain is not aperiodic, we shall still speak of geometric ergodicity if its restriction to its cyclic classes are.

\begin{lemma}\label{lemma:geom-erg}
Let $T$ be $(d_1,d_2)$-biregular tree with $d_1,d_2\geq 3$ and $\Gamma $ a geometrically finite tree lattice. Then the associated Markov chain $M_n$ is geometrically ergodic.
\end{lemma}

\begin{proof}
Let $F$ be the compact part of $Q$. For convenience of notation, we will assume $d_1 = d_2$. Let $q=d_1+1$. 

Recall that for $e\in EQ$, $|e|=d(\partial_1(e), F)$. The edge $e$ is said to point toward the finite part if $d(\partial_1(e),F)>d(\partial_0(e),F)$.

We define the  function $V:EQ \to [1,\infty)$  by 

\[ 
V(e) = 
\begin{cases}
1 & \text{if } e\in EF, \\
q^{0.1 |e|} & \text{if } e \notin EF, \text{and points toward the finite part} \\
\frac{1}{2^{|e|}}q^{0.9 |e|}  & \text{otherwise}.
\end{cases}
\]

We claim that 
$V$ satisfies the drift criteria \eqref{eq:driftcrit} with $\beta = q^{-0.1}$ and $b=q^5$.

Recall that we have positive transition probabilities only among neighboring edges in $EQ$. For $e\in EQ \setminus EF$,  the edge $e$ belongs to a Nagao ray. 
If $e$ is oriented toward the finite part and $|e|>5$, $PV(e) = V(f)$, where $|f| = |e|-1$. Hence, 
\[  PV(e) = q^{0.1 (|e|-1)} \leq q^{-0.1} V(e). \]

If $e$ is oriented toward the cusp, the transition probabilities are $1/q$ to jump one step further away from $EF$ to edge pointing toward the cusp and $q-1/q$ to get one step closer and point toward the finite part (see Example \ref{ex:chain.on.nagao}). In other words, for each edge $e$ with $|e|>5$  we have

\begin{equation*}
\begin{aligned}
 PV(e) &= \frac{1}{q} \cdot \frac{ q^{0.9(|e|+1)}}{2^{|e|+1}}  + \frac{q-1}{q} \cdot q^{0.1(|e|-1)} \\
 &   \leq  \frac{1}{2} \cdot q^{-0.1} \cdot  \frac{q^{0.9|e|}}{2^{|e|}} + q^{0.1(|e|-1)} \\
& \leq q^{-0.1} \cdot  \frac{q^{0.9|e|}}{2^{|e|}}  =  q^{-0.1} V(e).
\end{aligned}
\end{equation*}

The last inequality holds since for any $q\geq 4$ and $|e|>5$
\[  q^{0.1(|e|-1)} \leq \frac{1}{2} \cdot q^{-0.1} \cdot  \frac{q^{0.9|e|}}{2^{|e|}}. \]

The lemma follows by letting $K$ be the finite set of  edges with $|e|\leq 5$ (this also contains $EF$).
\end{proof}

Finally, the description of limit measures $\mu_j$'s in the paragraph following the statement of Theorem \ref{thm.dist.spheres} follows from the proof \S \ref{subsec.proof.of.C} and Lemma \ref{lemma.period} which says that the period of the Markov chain $M_n$ is always two so that the Dirac masses constituting each distribution $D_{\tilde{v}}$ all belong to a single cyclic class.


\begin{remark}\label{rk.equidist.in.X}
In the context of homogeneous dynamics, inequalities of type \eqref{eq:driftcrit} are often referred to as Margulis inequalities. They were first used in the work of Eskin--Margulis--Mozes \cite{eskin-margulis-mozes} and Eskin--Margulis \cite{eskin-margulis}. After we completed the first version of this article, for horospherical averages on lattice quotients of real semisimple groups, using linear representations, Katz \cite{katz} proved Margulis inequalities to establish quantitative non-divergence of horospherical averages (as in Lemma \ref{lemma:geom-erg}). Combining this with a spectral gap, he also deduced an equidistribution result (as Theorem \ref{thm.equidist} but) with rate depending, among others, on certain diophantine parameters of the starting point $x \in G/\Gamma$ (cf.~Remark \ref{rk.intro.dioph}). For $\PSL_2(\mathbb{R})$-quotients, more precise estimates  were obtained earlier by Flaminio--Forni \cite{flaminio-forni} and  Str\"{o}mbergsson \cite{strombergsson}, exploiting, among others, (unitary) representation theory of $\PSL_2(\mathbb{R})$.
\end{remark}

\begin{remark}
We remark that the family of lattices for which the associated Markov chain is geometrically ergodic and, consequently, for which part 3.\ of Theorem \ref{thm.dist.spheres} holds, contains many non-geometrically finite lattices. For example, the lattice associated with the edge-indexed graph from Fig.~\ref{fig:2q-ray} is such an example, with similar Foster--Lyapunov function $V(x)$ to the one in proof of Lemma \ref{lemma:geom-erg}. 
\end{remark}

\subsection{Proof of Theorem \ref{thm.counting}}
Let $\Gamma$ be a geometrically finite lattice in $\Aut(T)=:G$, denote by $m$ a Haar measure on $G$ and let $m_X$ be the induced $G$-invariant finite measure on $G/\Gamma$ by choice of a Borel fundamental domain in $G$. Denote by $S_T(R)$ the cardinality of the sphere of radius $R$ around $\tilde{o}$ in $T$ and $o= \pi(\tilde{o})$. As before, $\pi$ is the natural projection $VT \to VQ$ and $\rho_n$ denotes the normalized probability measure on the sphere of radius $n$ on $VQ$ with center $o$. Recall that $G$ has precisely two orbits on $VT$  and it acts transitively on the set of vertices of $T$ that are of even distance to each other, so that for every $\gamma \in \Gamma$, $2| d(\gamma \tilde{o},\tilde{o})$. For every $R \in \mathbb{N}$, we have
\begin{equation}\label{eq.rewrite.number}
 N(2R) = \sum_{n\leq R} S_T(2n) \cdot \pi_* \rho_{2n}(o)  \cdot | \Gamma \cap G_{\tilde{o}}| .    
\end{equation}

Thanks to $(3)$ of Theorem \ref{thm.dist.spheres} (see also the paragraph following that theorem), for some constant $r>1$, we have

\begin{equation}\label{eq.sphere.to.m}
|\pi_\ast \rho_{2n}(o)-\frac{1}{m_X(X)} \proj_\ast m_X(o)|=o(r^{-n}).
\end{equation}

On the other hand, in \eqref{eq.sphere.to.m}, the term $\proj_\ast m_X(o)$ can be rewritten as:
\begin{equation}\label{eq.rhs}
\proj_\ast m_X(o)=m_X(\proj^{-1}(o))=m_X(G_{\tilde{o}}\Gamma)=\frac{1}{|G_{\tilde{o}} \cap \Gamma|} m(G_{\tilde{o}}).
\end{equation}
Plugging \eqref{eq.rhs} and \eqref{eq.sphere.to.m} in \eqref{eq.rewrite.number} yields the desired statement.

To see the alternative expression of the main term $\frac{m(G_{\tilde{o}})}{m_X(X)}$ as expressed after the statement of Theorem \ref{thm.counting}, observe first that it follows by unimodularity of $\Aut(T)$ that for any two vertices $\tilde{v},\tilde{w} \in VT$ with $2 | d(\tilde{v},\tilde{w})$, we have $m(G_{\tilde{v}})=m(G_{\tilde{w}})$. Now fixing a lift $\tilde{v}$ for every $v \in VQ$ with $2 | d(o,v)$ and an element $g_v$ such that $g_v \tilde{v}= \tilde{o}$, we have

\begin{equation}\label{eq.one.is}
\begin{aligned}
m_X(X)& = \sum_{\underset{2| d(v,o)}{v \in VQ}} m_X(\proj^{-1}(v))=\sum_{\underset{2| d(v,o)}{v \in VQ}}m_X(G_{\tilde{o}}g_v \Gamma)=\sum_{\underset{2| d(v,o)}{v \in VQ}}m_X(G_{\tilde{v}}\Gamma)\\ &= \sum_{\underset{2| d(v,o)}{v \in VQ}} \frac{1}{|\Gamma \cap G_{\tilde{v}}|}m(G_{\tilde{v}})=m(G_{\tilde{o}})\sum_{\underset{2| d(v,o)}{v \in VQ}} \frac{1}{|\Gamma \cap G_{\tilde{v}}|}
\end{aligned}
\end{equation}
and get that the main term $\frac{m(G_{\tilde{o}})}{m_X(X)}$ is equal to $(\sum_{\underset{2| d(v,o)}{v \in VQ}} \frac{1}{|\Gamma \cap G_{\tilde{v}}|})^{-1}$.


\begin{thebibliography}{99}

\bibitem{Amann} O.  Amann, ``Group  of  tree-automorphisms  and  their  unitary  representations'',  ETH  Z\"{u}rich, 2003. PhD thesis.


\bibitem{bass-kulkarni}
Bass, H. and Kulkarni, R. ``Uniform tree lattices'', Journal of the American Mathematical Society 3, no. 4, (1990): 843--902.

\bibitem{bass-lubotzky} 
 Bass, H. and Lubotzky, A. ``Tree Lattices", Progr. Math., 176 (2001).
 
\bibitem{baxendale} Baxendale, P. H. "Renewal theory and computable convergence rates for geometrically ergodic Markov chains", The Annals of Applied Probability 15, no. 1B (2005): 700--738.


\bibitem{bekka-lubotzky}  Bekka, M. B. and  Lubotzky, A.  ``Lattices with and without spectral gap", Groups Geometry and Dynamics, 5 (2011): 251--264.

\bibitem{bekka-mayer}
Bekka, M. B. and Mayer, M. ``Ergodic theory and topological dynamics of group actions on homogeneous spaces", London Mathematical Society Lecture Note Series, vol. 269, Cambridge University Press, Cambridge, (2000)

\bibitem{benoist-quint.book}
Benoist, Y., and  Quint, J.-F.\ ``Random walks on reductive groups", In Random Walks on Reductive Groups, pp. 153--167. Springer, Cham, 2016.



\bibitem{bruhat-tits}  Bruhat, F. and Tits, J. ``Groupes r\'{e}ductifs sur un corps local: I. Donn\'{e}es radicielles
valu\'{e}es", Publ. Math. I.H.E.S. 41 (1972): 5--251.

\bibitem{bm-cat} Burger, M. and Mozes,  S. ``CAT(-1)-spaces, divergence groups and their commensurators", J. Amer. Math. Soc. 9, no. 1, (1996): 57--93


\bibitem{burger-mozes.local} 
Burger, M. and Mozes, S. ``Groups acting on trees: from local to global structure", Inst. Hautes \'{E}tudes Sci. Publ. Math., 92 (2000): 113--150.

\bibitem{burger-mozes.product} Burger, M. and Mozes, S. ``Lattices in product of trees", Inst. Hautes \'{E}tudes Sci. Publ. Math., 92 (2000): 151--194.


\bibitem{caprace-reid-willis}
Caprace, P-E., Reid, C. D.  and  Willis, G. A. ``Locally normal subgroups of totally disconnected groups. Part II: compactly generated simple groups", Forum of Mathematics, Sigma, vol 5 (2017). 


\bibitem{CFS} Ciobotaru, C., Finkelshtein, V. and Sert, C. "Measure rigidity for horospherical subgroups of groups acting on regular trees." arXiv preprint arXiv:1902.01300 (2019) (to appear in IMRN).

\bibitem{dani.recurrence} Dani, S.G. ``On orbits of unipotent 
ows on homogeneous spaces", Ergod. Th. and
Dynam. Syst. 4, (1984): 25--34 


\bibitem{dani.classification}
Dani, S. G.  ``Invariant  measures  and  minimal  sets  of  horospherical  flows", Invent. Math.64, no. 2, (1981): 357--385.

\bibitem{dani-margulis} Dani, S. G. and G. A. Margulis. ``Asymptotic behaviour of trajectories of unipotent flows on homogeneous spaces", In Proceedings of the Indian Academy of Sciences-Mathematical Sciences, vol. 101, no. 1, pp. 1-17. Springer India, 1991.

\bibitem{dani-smillie} Dani, S.G. and Smillie, J. ``Uniform distribution of horocycle orbits for Fuchsian Groups", Duke Math. J., 51, no. 1, (1984): 185-194

\bibitem{dedecker-gouezel} Dedecker, J. and Gou\"{e}zel, S. "Subgaussian concentration inequalities for geometrically ergodic Markov chains", Electronic Communications in Probability 20 (2015), no. 64, 12p.

\bibitem{delsarte} Delsarte, J., ``Sur le gitter fuchsien", C. R. Acad. Sci. Paris, 214, (1942): 147--179.

\bibitem{DMPS} Douc, R., Moulines, E., Priouret, P. and Soulier, P. ``Markov chains", Springer, 2018.


\bibitem{rudnick-duke-sarnak}
Duke, W., Rudnick,  Z., and  Sarnak P. ``Density of integer points on affine homogeneous varieties",  \textit{Duke Math. J.}, 71(1), (1993): 143--179.

\bibitem{ellis-perrizo}
Ellis, R. and Perrizo, W. ``Unique ergodicity of flows on homogeneous spaces", Israel J. Math, 29, 2-3 (1978): 276--284.

\bibitem{eskin-margulis}
Eskin, A. and Margulis, G. ``Recurrence properties of random walks on finite volume homogeneous manifolds'', In Random Walks and Geometry, pages 431-–444. Walter de Gruyter, Berlin, 2004.

\bibitem{eskin-margulis-mozes}
Eskin, A., Gregory, M, and Mozes, S. ``Upper bounds and asymptotics in a quantitative version of the Oppenheim conjecture", Annals of mathematics 147, no. 1 (1998): 93--141.

\bibitem{eskin-mcmullen} Eskin, A. and McMullen, C. ``Mixing, counting, and equidistribution in Lie groups", Duke Mathematical Journal, No 1 vol 71, (1993): 181--209.

\bibitem{fishman-simmons-urbanski}
Fishman, L. David S. and Urbański, M. ``Diophantine approximation and the geometry of limit sets in Gromov hyperbolic metric spaces", Vol. 254, no. 1215. American Mathematical Society, 2018.


\bibitem{flaminio-forni} Flaminio, L. and Giovanni F. ``Invariant distributions and time averages for horocycle flows'', Duke Mathematical Journal 119, no. 3 (2003): 465-526.

\bibitem{gorodnik-nevo} Gorodnik, A., and Nevo, A. ``Counting lattice points", J. Reine Angew. Math., 663, (2012): 127-176


\bibitem{ghosh} Ghosh, A. "Metric Diophantine approximation over a local field of positive characteristic", Journal of Number Theory 124, no. 2 (2007): 454--469.

\bibitem{hedlund} Hedlund, G. A. ``Fuchsian  groups  and  transitive  horocycles", Duke Math. J.2, no. 3, (1936): 530--542.

\bibitem{katz} Katz, A. ``Margulis' inequality for translates of horospherical orbits and application to equidistribution'', preprint.


\bibitem{kwon} Kwon, S. "Effective mixing and counting in Bruhat--Tits trees", Ergodic Theory and Dynamical Systems 38, no. 1 (2018): 257--283.




\bibitem{LPW} Levin, D. A., and Peres, Y. ``Markov chains and mixing times", Vol. 107. American Mathematical Soc., 2017.

\bibitem{lindenstrauss} Lindenstrauss, E. ``Pointwise theorems for amenable groups", Invent. Math.146, no. 2, (2001): 259--295.

\bibitem{lubotzky.gafa} Lubotzky, A. ``Lattices in rank one Lie groups over local fields", Geometric $\&$ Functional Analysis GAFA" 1(4), (1991): 405--431,

\bibitem{lubotzky-mozes} Lubotzky, A. and  Mozes, S. ``Asymptotic  properties  of  unitary  representations  of  tree  auto-morphisms", Harmonic analysis and discrete potential theory, Springer, Boston, MA., 1992.

\bibitem{meyn-tweedie} Meyn, S.P., and Tweedie L. R. ``Markov chains and stochastic stability", Springer Science $\&$ Business Media, 2012.

\bibitem{mohammadi} Mohammadi, A. ``Measures invariant under horospherical subgroups in positive characteristic", J. Mod. Dyn. 5, no. 2 (2011): 237--254.

\bibitem{mozes-shah} Mozes, S. and Shah, N. ``On the space of ergodic invariant measures of unipotent flows", Ergodic theory and dynamical systems, 15(1), (1995): 149--159.

\bibitem{nagao}
Nagao, H. ``On $\GL (2,\mathrm{K} [x])$", J. Inst. Osaka City Univ., 10, 2  (1959): 117--121.


\bibitem{paulin.cont.frac}
Paulin, F. "Groupe modulaire, fractions continues et approximation diophantienne en caractéristique p", Geometriae Dedicata 95, no. 1 (2002): 65--85.
	
\bibitem{paulin.geo.fin}
Paulin, F. ``Groupes g\'{e}om\'{e}triquement finis d'automorphismes d'arbres et approximation diophantienne dans les arbres", Manuscripta Math., 113, 1  (2004): 1--23.

\bibitem{raghunathan.geofin} Raghunathan, M.S. ``Discrete subgroups of algebraic groups over local fields of positive characteristics", Proc. Indian Acad. Sci. (Math. Scl.) 99 (1989): 127--146. 


\bibitem{randol}
Randol, B. ``The behavior under projection of dilating sets in a covering
space", Trans. Amer. Math. Soc., 285(2), (1984):855--859.

\bibitem{ratner.class} Ratner, M. ``On measure rigidity of unipotent subgroups of semisimple groups", Acta Math., 165 no.3-4, (1990): 229--309.

\bibitem{ratner.equidist} Ratner, M. ``Raghunathan’s topological conjecture and distributions of unipotent flows", Duke Math. J.63, (1991): 235--280

\bibitem{ratner:p-adic} Ratner, M. ``Raghunathan's conjectures for p-adic Lie groups", International Mathematics Research Notices 1993, no. 5 (1993): 141--146.

\bibitem{ratner.rpadic} Ratner, M. ``Raghunathan’s conjectures for Cartesian products of real and p-adic Lie groups", Duke Math. J.77 (2), (1995): 275--382.


\bibitem{roblin} Roblin, T. ``Ergodicit\'{e} et  \'{e}quidistribution en courbure n\'{e}gative", M\'{e}m. Soc.  Math. Fr. (95):vi+96, 2003


\bibitem{bass-serre}
Serre, J.-P. ``Arbres, amalgames, $\SL$2: cours au Coll{\`e}ge de France", R\'{e}dig\'{e} avec la collaboration de Hyman Bass; Soc. Math. France, (1977).

\bibitem{shah.equidist} Shah, N. A. ``Limit distributions of polynomial trajectories on homogeneous spaces", Duke Math. J., vol. 75 (3), (1994): 711--732.

\bibitem{shah-spheres} 
Shah, N. A. ``Limiting distributions of curves under geodesic flow on hyperbolic manifolds", Duke Math. J., 148 no. 2, (2009): 251--279.

\bibitem{strombergsson}
Strömbergsson, A. ``On the deviation of ergodic averages for horocycle flow'', Journal of Modern Dynamics 7, no. 2 (2013): 291--328.

\bibitem{tamagawa} Tamagawa, T. ``On discrete subgroups of $p$-adic algebraic groups" in Arithmetic Algebraic Geometry, Harper and Row, NewYork, (1965): 11--17.

\bibitem{tits}
Tits, J. ``Sur le groupe des automorphismes d'un arbre" in Essays on topology and related topics (M\'{e}moires d\'{e}di\'{e}s \`{a} Georges de Rham) Springer New York, (1970): 188--211.

\bibitem{vatsal} Vatsal, V. ``Uniform distribution of Heegner points", Inventiones mathematicae, 148(1), (2002): 1--46.






\end{thebibliography}
\end{document}